\documentclass[twoside]{article}
\usepackage[cp1251]{inputenc}
\usepackage[T2A]{fontenc}
\usepackage{array}
\usepackage{amssymb}
\usepackage{amsmath}
\usepackage{amsthm}
\usepackage{latexsym}
\usepackage{indentfirst}
\usepackage{bm}
\usepackage{enumerate}
\usepackage[dvips]{graphicx}
\usepackage{epsf}
\usepackage{wrapfig}
\usepackage{euscript}
\usepackage{indentfirst}
\usepackage[english,russian]{babel}
\usepackage{textcomp}
\newtheorem{theorem}{Theorem}
\newtheorem{lemma}{Lemma}
\newtheorem{definition}{Definition}

\newtheorem{corollary}{Corollary}
\newcommand{\co}{{\hspace{0.25mm}:\hspace{0.25mm}}}
\begin{document}
{\selectlanguage{english}
\binoppenalty = 10000 %
\relpenalty   = 10000 %

\pagestyle{headings} \makeatletter
\renewcommand{\@evenhead}{\raisebox{0pt}[\headheight][0pt]{\vbox{\hbox to\textwidth{\thepage\hfill \strut {\small Grigory. K. Olkhovikov}}\hrule}}}
\renewcommand{\@oddhead}{\raisebox{0pt}[\headheight][0pt]{\vbox{\hbox to\textwidth{{Completeness for implicit jstit logic}\hfill \strut\thepage}\hrule}}}
\makeatother

\title{A completeness result for implicit justification stit logic}
\author{Grigory K. Olkhovikov}
\date{}
\maketitle
\begin{quote}
\textbf{Abstract}. We present a completeness result for the
implicit fragment of justification stit logic introduced in
\cite{OLWA}. Although this fragment allows for no strongly
complete axiomatization, we show that a restricted form of strong
completeness (subsuming weak completeness) is available, as well
as deduce a version of restricted compactness property.
\end{quote}

\begin{quote}
stit logic, justification logic, completeness, compactness
\end{quote}

\section{Introduction}

Basic justification stit (or jstit, for short) logic was
introduced in \cite{OLWA} as an environment for analysis of
doxastic actions related to proving activity within a somewhat
idealized community of agents, combining expressive means of stit
logic by N. Belnap et al. \cite{belnap2001facing} with those of
justification logic by S. Artemov et al. \cite{ArtemovN05}. This
logic, therefore, retains the full set of expressive means of the
two above-mentioned logics and introduces some new expressive
means on top of them. These new expressive means were called in
\cite{OLWA} proving modalities and they capture different modes in
which one can speak about proving activity of an agent. The
general idea behind jstit logic is that one gets a right
classification of such modes if one intersects the distinction
between agentive and factual (aka moment-determinate) events
developed in stit logic with the distinction between explicit and
implicit modes of knowledge which is central to justification
logic. The first distinction, when applied to proofs, corresponds
to a well-known philosophical discussion of proofs-as-objects vs
proofs-as-acts. One refers to a proof-as-act when one says that
agent $j$ proves some proposition $A$, but one refers to a
proof-as-object when saying that $A$ was proved. While doing that,
one can either simply say that $A$ was proved, or add that $A$ was
proved by some proof $t$; and the difference between these two
modes of speaking is exactly the difference between implicit and
explicit reference to proofs. All in all this gives us the
following classification of proving modalities:

\begin{center}
 \begin{tabular}{|c||c|c|}
 \hline
  & Agentive & Moment-determinate \\
 \hline\hline
 Explicit & $j$ proves $A$ by $t$ & $A$ has been proven by $t$ \\
 & $Prove(j, t, A)$& $Proven(t,A)$\\
 \hline
 Implicit & $j$ proves $A$ & $A$ has been proven \\
 &$Prove(j, A)$& $Proven(A)$\\
 \hline
\end{tabular}
\end{center}

In \cite{OLWA} the semantics of these modalities was presented and
informally motivated in some detail. However, in the present
paper, we are going to look into one fragment of basic jstit logic
rather than the full system. The reason for this is the relatively
high level of complexity of the full basic jstit logic. The
fragment in question is, in fact, the basic jstit logic without
the two explicit proving modalities given in the first row of
table above. The resulting restricted system, therefore, features
the full set of expressive means inherited from justification
logic and stit logic plus the two implicit modalities,
$Prove(j,A)$ and $Proven(A)$. For the same reason (i.e. keeping
the complexity down), we also use a slightly simplified version of
the semantics introduced in \cite{OLWA} to interpret this logic.

The resulting system, which we will call the implicit jstit logic,
still allows for an analysis of the interplay between
proofs-as-acts and proofs-as-objects, although it limits the
format of such an analysis to some extent and also zeros out the
interplay between implicit and explicit modes of speech. But even
this restricted logic, has, as will be shown below, a challenging
degree of complexity, which makes the problem of axiomatizing it
both interesting and non-trivial.

The present paper is devoted to solving this exact problem. Its
layout is as follows. In Section \ref{basic} we define the
language and the semantics of the logic at hand. We also show some
features of implicit jstit logic, which limit the power and the
scope of possible completeness results, namely, the failure of
compactness and finite model properties. The latter fails in a
rather strong form; as a result, one cannot impose any finite
bound not only on the overall size of a model satisfying a given
formula, but also on the length of histories in such a model. The
failure of compactness also means that one cannot have a strongly
complete axiomatization for this logic while retaining a finitary
notion of proof.

Despite all these challenges, however, it turns out that with
implicit jstit logic one can do much better than just weak
completeness; in fact, our main result is much closer to the
strong completeness and only differs from the latter in that some
restrictions are imposed on proof variables occurring in a given
set of formulas. The exact formulation of this result is given in
Section \ref{axioms}, where we also define the axiom system which
displays this exact degree of completeness w.r.t. implicit jstit
logic. We immediately show this system to be sound w.r.t. the
semantics introduced in Section \ref{basic}, and we end the
section by proving a number of theorems in the system.

Section \ref{canonicalmodel} then contains the bulk of technical
work necessary for the completeness theorem. It gives a stepwise
construction and adequacy check for all the numerous components of
the canonical model and ends with a proof of a truth lemma.
Section \ref{main} then reaps the fruits of the hard work done in
Section \ref{canonicalmodel}, giving a concise proof of the
completeness result and drawing some quick corollaries including
the weak completeness theorem and a restricted form of compactness
property. Then follows Section \ref{conclusion}, giving some
conclusions and drafting directions for future work.

In what follows we will be assuming, due to space limitations, a
basic acquaintance with both stit logic and justification logic.
We recommend to peruse \cite{sep-logic-justification} for a quick
introduction into the basics of stit logic, and \cite[Ch.
2]{horty2001agency} for the same w.r.t. justification logic.

\section{Basic definitions and notation}\label{basic}

We fix some preliminaries. First we choose a finite set $Ag$
disjoint from all the other sets to be defined below. Individual
agents from this set will be denoted by letters $i$ and $j$. Then
we fix countably infinite sets $PVar$ of proof variables (denoted
by $x,y,z,w,u$) and $PConst$ of proof constants (denoted by
$a,b,c,d$). When needed, subscripts and superscripts will be used
with the above notations or any other notations to be introduced
in this paper. Set $Pol$ of proof polynomials is then defined by
the following BNF:
$$
t := x \mid c \mid s + t \mid s \times t \mid !t,
$$
with $x \in PVar$, $c \in PConst$, and $s,t$ ranging over elements
of $Pol$. In the above definition $+$ stands for the \emph{sum} of
proofs, $\times$ denotes \emph{application} of its left argument
to the right one, and $!$ denotes the so-called
\emph{proof-checker}, so that $!t$ checks the correctness of proof
$t$.

In order to define the set $Form$ of formulas we fix a countably
infinite set $Var$ of propositional variables to be denoted by
letters $p,q,r,s$. Formulas themselves will be denoted by letters
$A,B,C,D$, and the definition of $Form$ is supplied by the
following BNF:
\begin{align*}
A := p \mid A \wedge B \mid \neg A \mid [j]A \mid \Box A \mid t\co
A \mid KA \mid Prove(j, A) \mid Proven(A),
\end{align*}
with $p \in Var$, $j \in Ag$ and $t \in Pol$.

It is clear from the above definition of $Form$ that we are
considering a version of modal propositional language. As for the
informal interpretations of modalities, $[j]A$ is the so-called
cstit action modality and $\Box$ is historical necessity modality,
both modailities are borrowed from stit logic. The next two
modailities, $KA$ and $t\co A$, come from justification logic and
the latter is interpreted as ``$t$ proves $A$'', whereas the
former is the strong epistemic modality ``$A$ is
known''.\footnote{Perhaps, ``$A$ is provable'' will be an even
better reading.} The two remaining modalities, $Prove(j, A)$ and
$Proven(A)$ are implicit modalities related to the proving
activity of agents and their informal interpretation was
considered in Section 1.

We assume $\Diamond$, $\langle K \rangle$, and $\langle j \rangle$ for a $j \in Ag$
as notations for the dual modalities of $\Box$, K and $[j]$, respectively.

For the language at hand, we assume the following semantics. A
jstit model is a structure
$$
\mathcal{M} = \langle Tree, \leq, Choice, Act, R, \mathcal{E},
V\rangle
$$
such that:
\begin{itemize}
\item $Tree$ is a non-empty set. Elements of $Tree$ are called
\emph{moments}.

\item $\leq$ is a partial order on $Tree$ for which a temporal
interpretation is assumed.

\item $Hist$ is the set of maximal chains in $Tree$ w.r.t. $\leq$.
Since $Hist$ is completely determined by $Tree$ and $\leq$, it is
not included into the structure of a model as a separate
component. Elements of $Hist$ are called \emph{histories}. The set
of histories containing a given moment $m$ will be denoted $H_m$.
The following set:
$$
MH(\mathcal{M}) = \{ (m,h)\mid m \in Tree,\, h \in H_m \},
$$
called the set of \emph{moment-history pairs}, will be used to
evaluate formulas of the above language.

\item $Choice$ is a function mapping $Tree \times Agent$ into
$2^{2^{Hist}}$ in such a way that for any given $j \in Agent$ and
$m \in Tree$ we have as $Choice(m,j)$ (to be denoted as
$Choice^m_j$ below) a partition of $H_m$. For a given $h \in H_m$
we will denote by $Choice^m_j(h)$ the element of partition
$Choice^m_j$ containing $h$.

\item $Act$ is a function mapping $MH(\mathcal{M})$ into
$2^{Pol}$.

\item $R$ is a pre-order on $Tree$ called epistemic
accessibility.

\item $\mathcal{E}$ is a function mapping $Tree \times Pol$ into
$2^{Form}$.

\item $V$ is an evaluation function, mapping the set $Var$ into
$2^{MH(\mathcal{M})}$.
\end{itemize}

However, not all structures of the above described type are
admitted as jstit models. A number of additional restrictions
needs to be satisfied. More precisely, we assume satisfaction of
the following constraints:

\begin{enumerate}
\item Historical connection:
$$
(\forall m,m_1 \in Tree)(\exists m_2
\in Tree)(m_2 \leq m \wedge m_2 \leq m_1).
$$

\item No backward branching:
$$
(\forall m,m_1,m_2 \in Tree)((m_1 \leq m \wedge m_2 \leq m) \to
(m_1 \leq m_2 \vee m_2 \leq m_1)).
$$

\item No choice between undivided histories:
$$
(\forall m,m' \in Tree)(\forall h,h' \in H_m)(m < m' \wedge m' \in
h \cap h' \to Choice^m_j(h) = Choice^m_j(h'))
$$
for every $j \in Agent$.

\item Independence of agents:
$$
(\forall m\in Tree)(\forall f:Agent \to 2^{H_m})((\forall j \in
Agent)(f(j) \in Choice^m_j) \Rightarrow \bigcap_{j \in Agent}f(j)
\neq \emptyset).
$$

\item Monotonicity of evidence:
$$
(\forall t \in Pol)(\forall m,m' \in Tree)(R(m,m') \Rightarrow
\mathcal{E}(m,t) \subseteq \mathcal{E}(m',t)).
$$

\item Evidence closure properties. For arbitrary $m \in Tree$,
$s,t \in Pol$ and $A, B \in Form$ it is assumed that:
\begin{enumerate}
\item $A \to B \in \mathcal{E}(m,s) \wedge A \in \mathcal{E}(m,t)
\Rightarrow B \in \mathcal{E}(m,s\times t)$;

\item $\mathcal{E}(m,s) \cup \mathcal{E}(m,t) \subseteq
\mathcal{E}(m,s + t)$.
\item $A \in \mathcal{E}(m,t) \Rightarrow
t:A \in \mathcal{E}(m,!t)$;
\end{enumerate}

\item Expansion of presented proofs:
$$
(\forall m,m' \in Tree)(m' < m \Rightarrow \forall h \in H_m
(Act(m',h) \subseteq Act(m,h))).
$$

\item No new proofs guaranteed:
$$
(\forall m \in Tree)(\bigcap_{h \in H_m}(Act(m,h)) \subseteq
\bigcup_{m' < m, h \in H_m}(Act(m',h))).
$$

\item Presenting a new proof makes histories divide:
$$
(\forall m \in Tree)(\forall h,h' \in H_m)(\exists m' > m(m' \in h
\cap h') \Rightarrow (Act(m,h) = Act(m,h'))).
$$

\item Future always matters:
$$
\leq \subseteq R.
$$

\item Presented proofs are epistemically transparent:
$$
(\forall m,m' \in Tree)(R(m,m') \Rightarrow (\bigcap_{h \in
H_m}(Act(m,h)) \subseteq \bigcap_{h' \in H_{m'}}(Act(m',h')))).
$$
\end{enumerate}
We offer some intuitive explanation for the above defined notion
of jstit model. Due to space limitations, we only explain the
intuitions behind jstit models very briefly, and we urge the
reader to consult \cite[Section 3]{OLWA} for a more comprehensive
explanations, whenever needed.

The components like $Tree$, $\leq$, $Choice$ and $V$ are inherited
from stit logic, whereas $R$ and $\mathcal{E}$ come from
justification logic. The only new component is $Act$. The
intuition behind the semantics is that $Ag$, our community of
agents, is engaged in proving activity and this proving activity
consists in making proof polynomials public within the community.
One can think of a group of researchers, assembled before a
whiteboard in a conference room and putting the proofs they
discover on this whiteboard. Function $Act$ gives out the current
state of this whiteboard at any given moment under any given
history. The whole situation is somewhat idealized in that we
assume that nothing ever gets erased from the whiteboard, that
there is always enough free space on it, and that the agents do
not send one another any private messages.

The numbered list of semantical constraints above then just builds
on these intuitions. Constraints $1$--$4$ are borrowed from stit
logic, constraints $5$ and $6$ are inherited from justification
logic. Constraint $7$ just says that nothing gets erased from the
whiteboard, constraint $8$ says a new proof cannot spring into
existence as a static (i.e. moment-determinate) feature of the
environment out of nothing, but rather has to come as a result (or
a by-product) of a previous activity. Constraint $9$ is just a
corollary to constraint $3$ in the richer environment of jstit
models, constraint $10$ says that the possible future of the given
moment is always epistemically relevant in this moment, and
constraint $11$ says that the community knows everything that has
firmly made its way onto the whiteboard.

For the members of $Form$, we will assume the following
inductively defined satisfaction relation. For every jstit model
$\mathcal{M} = \langle Tree, \leq, Choice, Act, R, \mathcal{E},
V\rangle$ and for every $(m,h) \in Pair_\mathcal{M}$ we stipulate
that:

\begin{align*}
&\mathcal{M}, m, h \models p \Leftrightarrow (m,h) \in
V(p);\\
&\mathcal{M}, m, h \models [j]A \Leftrightarrow (\forall h'
\in Choice^m_j(h))(\mathcal{M}, m, h' \models A);\\
&\mathcal{M}, m, h \models \Box A \Leftrightarrow (\forall h'
\in H_m)(\mathcal{M}, m, h' \models A);\\
&\mathcal{M}, m, h \models KA \Leftrightarrow \forall m'\forall
h'(R(m,m') \& h' \in H_{m'} \Rightarrow \mathcal{M}, m', h'
\models A);\\
&\mathcal{M}, m, h \models t\co A \Leftrightarrow A \in
\mathcal{E}(m,t) \& \mathcal{M}, m, h \models KA;\\
&\mathcal{M}, m, h \models Prove(j, A) \Leftrightarrow (\forall h'
\in Choice^m_j(h))(\exists t \in Act(m,h'))(\mathcal{M},m,h'
\models t\co A)
\&\\
&\qquad\qquad\qquad\qquad\qquad \&(\forall s \in Pol)(\exists
h'' \in H_m) (\mathcal{M},m,h \models s\co A \Rightarrow s \notin
Act(m,h''));\\
&\mathcal{M}, m, h \models Proven(A) \Leftrightarrow (\exists t
\in Pol)(\forall h' \in H_m) (t \in Act(m,h') \& \mathcal{M},
m, h \models t\co A)
\end{align*}

In the above clauses we assume that $p \in Var$; we also assume
standard clauses for Boolean connectives. Note that the
satisfaction clause for $Prove(j,A)$ consists of two conjuncts,
one stating that some proof of $A$ must be presented at every
history in a given choice cell, and the other saying that no proof
of $A$ is presented in all histories through the given moment.
These conjuncts show some similarity to the conjuncts in the
satisfaction clause for the \emph{dstit}
 operator, which are known in the existing literature under the names of \emph{positive} and \emph{negative}
 condition, respectively. Following this usage, we will
 name the first conjunct in the satisfaction clause for $Prove(j,A)$
 the positive condition for $Prove(j,A)$, and the second one
 the negative condition for $Prove(j,A)$. The intuitive motivation
 for the satisfaction clauses of $Prove(j,A)$ and $Proven(A)$ was
 worked out in detail in \cite{OLWA} and we do not dwell on it
 here.

 We further assume standard definitions for satisfiability and
 validity of formulas and sets of formulas in the presented
 semantics.

 Before we proceed to proving things about the defined system, we
 want to briefly comment on how the above semantics relates to the
 semantics introduced in \cite{OLWA}. The main difference is that
 the latter semantics uses two epistemic accessibility relations $R$ and
 $R_e$ with the constraint that $R
\subseteq R_e$, whereas in the jstit models as defined above one
only finds one such relation $R$, and this relation serves the
functions of both $R$ and $R_e$. Thus the semantics defined above
arises from the more general semantics presented in \cite{OLWA} as
a particular case with $R$ and $R_e$ being identified with one
another.

The exact import of this additional restriction on the semantics
presented in \cite{OLWA} is not yet clear. It is known that on the
level of pure justification logic identifying $R$ and $R_e$ does
not change the set of validities (see, e.g. \cite[Comment
6.5]{ArtemovN05}). Our tentative hypothesis would be, then, that
imposing $R = R_e$ in the richer context of jstit logic might be
just as irrelevant as it is in justification logic. However, we
have no proof of this hypothesis at the moment, so it stands as an
open problem.

The semantics just defined admits of no finitary strongly complete
system since it is not compact. Indeed, the set
$$
\{ Proven(p) \} \cup \{ \neg t\co p\mid t \in Pol \}
$$
is unsatisfiable, even though every finite subset of it can be
satisfied. Still, the main result of this paper shows that we can
do better than just weak completeness; in fact we can show that
also infinite consistent sets of formulas can be satisfied
provided that there is an infinite set of proof variables that do
not occur in those formulas. Thus we get something considerably
stronger than just weak completeness including also a restricted
form of the compactness theorem.

It is also worth noting that under the presented semantics some
satisfiable formulas cannot be satisfied over finite models. As an
example of this phenomenon, consider $K(\Diamond p \wedge
\Diamond\neg p)$. If $\mathcal{M}, m_1, h \models K(\Diamond p
\wedge \Diamond\neg p)$, then, by reflexivity of $R$, also
$\mathcal{M}, m_1, h \models \Diamond p \wedge \Diamond\neg p$,
which means that at least two different histories are running
through $m_1$ in $\mathcal{M}$. Therefore, $m_1$ cannot be a
$\leq$-maximal moment in $\mathcal{M}$, so that there is at least
one moment $m_2 \in h$ such that $m_1 < m_2$. By the future always
matters constraint we get then that $R(m_1, m_2)$, which, by
transitivity of $R$, means that we also have $\mathcal{M}, m_2, h
\models K(\Diamond p \wedge \Diamond\neg p)$. Iterating this
construction $\omega$ times, we get a countably infinite sequence
of moments along $h$:
$$
m_1 < m_2 <\ldots < m_n <\ldots,
$$
showing that the moments in these sequence are pairwise different
(by antisymmetry of $\leq$) and that $\mathcal{M}$ is consequently
an infinite model. Since $\mathcal{M}$ was chosen arbitrarily,
this shows that $K(\Diamond p \wedge \Diamond\neg p)$ cannot be
satisfied over finite jstit models. On the other hand, $K(\Diamond
p \wedge \Diamond\neg p)$ is clearly satisfiable when one allows
for infinite models. One can consider, for example, a jstit model
$\mathcal{M} = \langle Tree, \leq, Choice, Act, R, \mathcal{E},
V\rangle$ for a community $\{ j \}$ consisting of a single agent,
setting:
$$
Tree: = \{ (a_1,\ldots,a_n)\mid a_i \in \{ 0, 1 \}\textup{ for } i
\leq n \} \cup \{\Lambda \},
$$
where $\Lambda$ is the empty sequence;
$$
(a_1,\ldots,a_n) \leq (b_1,\ldots,b_k) \Leftrightarrow (n \leq k
\& (\forall i \leq n)(a_i = b_i)),\,\,Choice^m_j := H_m,\,\,
Act(m, h) = \emptyset
$$
for every $m \in Tree$ and $h \in H_m$;
$$
R := \leq, \mathcal{E}(m, t) := Form,
$$
for every $m \in Tree$ and $t \in Pol$;
$$
V(p) := \{ (m,h) \mid (m, 1) \in h \}, V(q) = \emptyset
$$
provided $q \in Var \setminus \{ p \}$. It is straightforward to
check then that with these settings we get that $\mathcal{M},
\Lambda, h \models K(\Diamond p \wedge \Diamond\neg p)$ for an
arbitrary history $h$ over $\mathcal{M}$.

Note also, that the same example shows that one cannot put a
finite bound on the length of histories in the models satisfying a
given formula, so that what one might have called a ``finite
history property'' which is satisfied, e.g., by the canonical
model of the logic of dstit operator (see \cite[Section
17C]{belnap2001facing} for the definition) also fails for the
implicit jstit logic.

\section{Axiomatic system and soundness}\label{axioms}

We consider the following set of axiomatic schemes:

\begin{align}
&\textup{A full set of axioms for classical propositional
logic}\label{A0}\tag{\text{A0}}\\
&\textup{$S5$ axioms for $\Box$ and $[j]$ for every $j \in
Agent$}\label{A1}\tag{\text{A1}}\\
&\Box A \to [j]A \textup{ for every }j \in Agent\label{A2}\tag{\text{A2}}\\
&(\Diamond[j_1]A_1 \wedge\ldots \wedge \Diamond[j_n]A_n) \to
\Diamond([j_1]A_1 \wedge\ldots \wedge[j_n]A_n)\label{A3}\tag{\text{A3}}\\
&(s\co(A \to B) \to (t\co A \to (s\times t)\co
B)\label{A4}\tag{\text{A4}}\\
&t\co A \to (!t\co(t\co A) \wedge KA)\label{A5}\tag{\text{A5}}\\
&(s\co A \vee t\co A) \to (s+t)\co A\label{A6}\tag{\text{A6}}\\
&\textup{$S4$ axioms for $K$}\label{A7}\tag{\text{A7}}\\
&KA \to \Box K\Box A\label{A8}\tag{\text{A8}}\\
&Prove(j, A) \to (\neg Proven(A) \wedge [j]Prove(j, A) \wedge KA)\label{A9}\tag{\text{A9}}\\
&\Box Prove(j, A) \to \Box Prove(i, A)\label{A10}\tag{\text{A10}}\\
&Proven(A) \to (KProven(A) \wedge KA)\label{A11}\tag{\text{A11}}\\
&\neg K(\bigvee^{n}_{l = 1}\langle K \rangle\Diamond Prove(j_{l},
A_{l}))\label{A12}\tag{\text{A12}}\\
&\neg Prove(j,A) \to \langle j \rangle (\bigwedge_{i \in Ag}\neg Prove(i,A))\label{A13}\tag{\text{A13}}
\end{align}

The assumption is that in \eqref{A3} $j_1,\ldots, j_n$ are
pairwise different.

To this set of axiom schemes we add the following rules of
inference:
\begin{align}
&A, A \to B \Rightarrow B;\label{R1}\tag{\text{R1}}\\
&A\Rightarrow KA;\label{R2}\tag{\text{R2}}\\
&\textup{If $A$ is an instance of (A0)--(A13) and $c \in Const$,
then infer $c\co A$;}\label{R3}\tag{\text{R3}}\\
&KA \to (\neg Proven(B_1) \vee\ldots \vee\neg
Proven(B_n)) \Rightarrow\notag\\
&\qquad\qquad\Rightarrow KA \to (\bigwedge_{j \in Ag}\neg
Prove(j,B_1) \vee\ldots \vee \bigwedge_{j \in Ag}\neg
Prove(j,B_n)).\label{R4}\tag{\text{R4}}
\end{align}

We call a jstit model $\mathcal{M} = \langle Tree, \leq, Choice,
Act, R, \mathcal{E}, V\rangle$ \emph{normal} iff the following
condition is satisfied:
\begin{align*}
(\forall c \in Const)(\forall m \in Tree)(\{ A \mid &A\text{ is a
substitution instance}\\
&\text{of one of the schemes among \eqref{A1}--\eqref{A13}}\}
\subseteq \mathcal{E}(m,c)).
\end{align*}

Our goal is now a restricted completeness theorem w.r.t. the class
of normal models. We start by establishing soundness, and we
precede the soundness theorem with the following rather
straightforward technical claim:
\begin{lemma}\label{determinate}
For every $A \in Form$ and every $t \in Pol$, all of the formulas
$\Box A$, $KA$, $t\co A$ and $Proven(A)$ are moment-determinate,
that is to say, if $\alpha \in \{ \Box A, KA, t\co A,Proven(A)
\}$, then for an arbitrary normal jstit model $\mathcal{M} =
\langle Tree, \leq, Choice, Act, R, \mathcal{E}, V\rangle$ and $m
\in Tree$, if $h, h' \in H_m$, then:
$$
\mathcal{M},m,h \models \alpha \Leftrightarrow \mathcal{M},m,h'
\models \alpha.
$$
Also, Boolean combinations of these formulas are
moment-determinate.
\end{lemma}
\begin{proof}
For $\alpha = \Box A$ and $\alpha = KA$ it suffices to note that
the semantical conditions for satisfaction of $KA$ and $\Box A$ at
a given $(m,h) \in MH(\mathcal{M})$ in a given $\mathcal{M}$ have
no free occurrences of $h$. When we turn, further, to the
corresponding condition for $t\co A$, the only free occurrence of
$h$ will be within the context $\mathcal{M}, m, h \models KA$
which was shown to be moment-determinate. Similarly, in the
satisfaction condition for $Proven(A)$ the only free occurrence of
$h$ is within a moment determinate context $\mathcal{M}, m, h
\models t\co A$.

Of course, Boolean combinations of moment-determinate formulas
must be moment-determinate, too.
\end{proof}

It follows from Lemma \ref{determinate}, that the truth of
moment-determinate formulas at a given moment-history pair only
depends on the moment, so that we might as well omit the histories
when discussing satisfaction of such formulas and write
$\mathcal{M}, m \models KA$ instead of $\mathcal{M}, m, h \models
KA$, etc.

Establishing soundness mostly reduces to a routine check that
every axiom is valid and that rules preserve validity. We treat
the less obvious cases in some detail:

\begin{theorem}\label{soundness}
Every instance of \eqref{A1}--\eqref{A13} is valid over the
class of normal jstit models. Every application of rules
\eqref{R1}--\eqref{R4} to formulas which are valid over the class of normal jstit models
yileds a formula which is valid over the class of normal jstit models.
\end{theorem}
\begin{proof}
First, note that if $\mathcal{M} = \langle Tree, \leq, Choice,
Act, R, \mathcal{E}, V\rangle$ is a normal jstit model, then
$\langle Tree, \leq, Choice, V\rangle$ is a model of stit logic.
Therefore, axioms \eqref{A0}--\eqref{A3}, which were copy-pasted
from the standard axiomatization of \emph{dstit} logic (see, e.g.
\cite[Ch. 17]{belnap2001facing}) must be valid. Second, note that
if $\mathcal{M} = \langle Tree, \leq, Choice, Act, R, \mathcal{E},
V\rangle$ is a normal jstit model, then $\langle Tree, \leq, Act,
R, \mathcal{E}\rangle$ is what is called in \cite[p.
1067]{ArtemovN05} a frame for a Fitting justification model with
the form of constant specification defined by
\eqref{R3}\footnote{But note, that in \cite{ArtemovN05} they do
not include $\mathcal{E}$ in justification frames; however, this
is of no consequence for the present setting.}. This means that
also all of the \eqref{A4}--\eqref{A7} must be valid, whereas
\eqref{R1}--\eqref{R3} must preserve validity. The validity of
other elements of the above-presented axiomatic system will be
motivated below in some detail. In what follows, $\mathcal{M} =
\langle Tree, \leq, Choice, Act, R, \mathcal{E}, V\rangle$ will
always stand for an arbitrary normal jstit model, and $(m,h)$ for
an arbitrary element of $MH(\mathcal{M})$.

As for \eqref{A8}, assume for \emph{reductio} that $\mathcal{M}, m
\models KA \wedge \neg\Box K\Box A$. Then $\mathcal{M}, m, h
\models KA$ and also $\mathcal{M}, m \not\models \Box K\Box A$.
The latter means that for some $h' \in H_m$ we have $\mathcal{M},
m \not\models K\Box A$. Therefore, there must be some $m' \in
Tree$ such that $R(m,m')$ and some $g \in H_{m'}$ such that
$\mathcal{M}, m', g \not\models \Box A$, whence for some $g' \in
H_{m'}$ we will have $\mathcal{M}, m', g' \not\models A$. Since
$R(m,m')$, this means that $KA$ must fail at $(m,h)$ in
$\mathcal{M}$, a contradiction.

We consider next \eqref{A9}. Assume that $Prove(j, A)$ is true at
$(m,h)$ in $\mathcal{M}$. Note that the negative condition for
$Prove(j, A)$ at $(m,h)$ is logically equivalent to the negation
of the satisfaction condition for $Proven(A)$, which means that
$\neg Proven(A)$ must be  true at $(m,h)$ in $\mathcal{M}$.
Further, since clearly $h \in Choice^m_j(h)$ and thus
$Choice^m_j(h)$ cannot be empty, it follows from the positive
condition for $Prove(j, A)$ that for some $t \in Pol$ we will have
$\mathcal{M},m,h \models t\co A$, and therefore, by validity of
\eqref{A5}, $\mathcal{M},m,h \models KA$. Finally, note that since
$Choice^m_j$ is a partition of $H_m$, then for any $h' \in
Choice^m_j(h)$, if $h'' \in Choice^m_j(h')$, then $h'' \in
Choice^m_j(h)$. Therefore, since the positive condition for
$Prove(j,A)$ is satisfied at $(m,h)$, there must be some $t \in
Pol$ such that both $t \in Act(m,h'')$ and $\mathcal{M},m\models
t\co A$. Therefore, the positive condition for $Prove(j,A)$ will
be satisfied at $(m,h')$ for every $h' \in Choice^m_j(h)$. As for
the negative condition, recall that it is equivalent to the
negation of the satisfaction condition for $Proven(j,A)$ and the
latter is, by Lemma \ref{determinate}, moment-determinate.
Therefore, the negative condition for $Prove(j,A)$ must be
moment-determinate as well, and, once satisfied at a given
$(m,h)$, it will be satisfied at every history through $m$.
Therefore, once we have  $Prove(j, A)$ true at $(m,h)$ in
$\mathcal{M}$, we must also have $\mathcal{M},m,h \models
[j]Prove(j,A)$.

The next axiom is \eqref{A10}. If $\Box Prove(j, A)$ is true at
$(m,h)$ in $\mathcal{M}$, this means that $Prove(j, A)$ is true at
$(m,h')$ in $\mathcal{M}$ for every $h' \in H_m$. Now, take an
arbitrary such $h'$. We know that the negative condition for
$Prove(i,A)$ is the same as for $Prove(j, A)$, and is therefore
satisfied at $(m,h')$. As for the positive condition, assume that
$h'' \in Choice^m_i(h')$. We know that $Prove(j, A)$ is true at
$(m,h'')$, therefore, since $h''$ is obviously in
$Choice^m_j(h'')$, for some $t \in Pol$ we must have both $t \in
Act(m,h'')$ and $\mathcal{M},m \models t\co A$. Thus the positive
condition for $Prove(i,A)$ at $(m,h')$ is satisfied as well. Since
$h'$ was chosen as an arbitrary history through $m$, this means
that $\Box Prove(i, A)$ must be satisfied at $(m,h)$ in
$\mathcal{M}$.

We now take up \eqref{A11}. If $Proven(A)$ is true at $m$ in
$\mathcal{M}$, then there is a $t \in Pol$ such that $t \in
\bigcap_{h \in H_m}Act(m,h)$ and $t\co A$ is true at $m$. By
validity of \eqref{A5}, we immediately get that $\mathcal{M},m
\models KA$. Further, the fact that $t\co A$ is true at $m$ means
that $A \in \mathcal{E}(m,t)$. Now, assume that $m' \in Tree$ is
such that $R(m,m')$. By the epistemic transparency of presented
proofs constraint we know that $t \in \bigcap_{h' \in
H_{m'}}Act(m',h')$. By monotonicity of evidence, we know that $A
\in \mathcal{E}(m',t)$. By the S4 reasoning for $K$ we know that
$\mathcal{M},m' \models KA$. Summing up, we must have $Proven(A)$
true at $m'$, and since $m'$ was chosen as an arbitrary
$R$-successor of $m$, this means that we also have $\mathcal{M},m
\models KProven(A)$.

To prove the validity of \eqref{A12} over the class of normal
jstit models, we proceed by induction on $n \geq 1$.

\emph{Basis}. $n = 1$. Assume, for \emph{reductio}, that
$\mathcal{M},m \models K\langle K\rangle\Diamond Prove(j_1,A_1)$.
Then, by validity of \eqref{A7}, $\mathcal{M},m \models \langle
K\rangle\Diamond Prove(j_1,A_1)$. Therefore, for some $m' \in
Tree$ such that $R(m,m')$, we must have $\mathcal{M},m' \models
\Diamond Prove(j_1,A_1)$. The latter, in turn, means that for some
$h' \in H_{m'}$ we will have $\mathcal{M},m', h' \models
Prove(j_1,A_1)$. We know then that $m'$ must have some
$<$-successors, where $<$ is the irreflexive companion of $\leq$
in $\mathcal{M}$. Indeed, if $m'$ were a $\leq$-maximal moment,
then we would have $H_{m'} = \{ h' \}$, that is to say, $h'$ would
be the only history passing through $m'$. But then, of course $h'
\in Choice^{m'}_{j_1}(h')$, therefore, for some $t \in Pol$ we
would have then both $t \in Act(m', h')$ and $\mathcal{M},m'
\models t\co A$ by the positive condition for $Prove(j_1,A_1)$ at
$(m',h')$. But then, given that $H_{m'} = \{ h' \}$, this would
mean that $t \in \bigcap_{g \in H_{m'}}Act(m',g)$ so that the
negative condition for $Prove(j_1,A_1)$ at $(m',h')$ would be
violated, contradicting our assumption that $\mathcal{M},m', h'
\models Prove(j_1,A_1)$.

Therefore, we can choose a moment $m''$ such that both $m'' > m'$
and $h'$ passes through $m''$; consider then $H_{m''}$. All the
histories passing through $m''$ are pairwise undivided at $m'$,
therefore, by the presenting a new proof makes histories divide
constraint we must have $Act(m', g) = Act(m', g')$ for any $g,g'
\in H_{m''}$. We also know that, since $\mathcal{M},m', h' \models
Prove(j_1,A_1)$, there must be a $t \in Pol$ such that $t \in
Act(m',h')$ and $\mathcal{M},m' \models t\co A$. Since $h' \in
H_{m''}$, this further means that $t \in \bigcap_{g \in
H_{m''}}Act(m', g)$. By the expansion of presented proofs
constraint, we may infer from the latter that $t \in \bigcap_{g
\in H_{m''}}Act(m'', g)$. By the future always matters constraint,
we know that, since $m' < m''$, then we must have $R(m',m'')$,
whence, given that $\mathcal{M},m' \models t\co A$, we must also
have $\mathcal{M},m'' \models t\co A$. Summing this up with $t \in
\bigcap_{g \in H_{m''}}Act(m'', g)$, we get that $\mathcal{M},m''
\models Proven(A)$, which, by \eqref{A11}, means that
$\mathcal{M},m'' \models KProven(A)$, whence further, by
\eqref{A8}, $\mathcal{M},m'' \models \Box K\Box Proven(A)$.
Validity of \eqref{A1} yields then $\mathcal{M},m'' \models K\Box
Proven(A)$. Note, further, that $Proven(A) \to \neg Prove(j,A)$
must be valid as a consequence of \eqref{A9}, and by S5 reasoning
for $\Box$ and S4 reasoning for $K$ we get from this the validity
of:
$$
K\Box Proven(A) \to K\Box\neg Prove(j,A).
$$
The latter means that $\mathcal{M},m'' \models K\Box\neg
Prove(j,A)$, and, pushing the negation outside, $\mathcal{M},m''
\models \neg\langle K\rangle\Diamond Prove(j,A)$. It remains then
to note we already established both $R(m,m')$ and $R(m,m'')$ so
that by transitivity of $R$ we get that $R(m,m'')$. Therefore, the
consequence that $\mathcal{M},m'' \models \neg\langle
K\rangle\Diamond Prove(j,A)$ turns out to be in contradiction with
our initial hypothesis that $\mathcal{M},m \models K\langle
K\rangle\Diamond Prove(j_1,A_1)$. The obtained contradiction shows
that we must have $\neg K\langle K\rangle\Diamond Prove(j_1,A_1)$
true throughout any given normal jstit model for any $A_1 \in
Form$ and $j_1 \in Ag$.

\emph{Induction step}. Assume that for a $k \geq 1$ the validity
of all instances of the scheme $\neg K(\bigvee^{k}_{l = 1}\langle
K \rangle\Diamond Prove(j_{l}, A_{l}))$ has been successfully
shown and assume that $n = k + 1$. Assume, further, that:
$$
\mathcal{M}, m \models K(\bigvee^{k + 1}_{l = 1}\langle K
\rangle\Diamond Prove(j_{l}, A_{l})).
$$
Then, by S4 reasoning for $K$, we know that
$$
\mathcal{M}, m \models \bigvee^{k + 1}_{l = 1}\langle K
\rangle\Diamond Prove(j_{l}, A_{l}),
$$
so that at least one of $\langle K \rangle\Diamond Prove(j_{l},
A_{l})$ must be true at $m$; suppose, wlog, that $l = 1$. Then,
arguing as in the base case, we find a moment $m''$ such that
$R(m,m'')$ and $\mathcal{M},m'' \models K\Box\neg Prove(j_1,A_1)$.
Applying to this S4 reasoning for $K$, we get further that
$\mathcal{M},m'' \models KK\Box\neg Prove(j_1,A_1)$, and, pushing
out the negation, that $\mathcal{M},m'' \models K\neg\langle
K\rangle\Diamond Prove(j_1,A_1)$. Since we have $R(m,m'')$, it
follows that we also have:
$$
\mathcal{M}, m'' \models K(\bigvee^{k + 1}_{l = 1}\langle K
\rangle\Diamond Prove(j_{l}, A_{l})).
$$
From the latter two facts, S4 reasoning for $K$ yields that:
$$
\mathcal{M}, m'' \models K(\bigvee^{k + 1}_{l = 2}\langle K
\rangle\Diamond Prove(j_{l}, A_{l})),
$$
contradicting the induction hypothesis. The obtained contradiction
shows the validity of \eqref{A12} for $n = k + 1$.

The last axiom is \eqref{A13}. So, assume that $\mathcal{M}, m, h
\models \neg Prove(j, A)$. We have to consider then two cases.

\emph{Case 1}. The negative condition for $Prove(j,A)$ fails at
$(m,h)$. Then we must have $\mathcal{M}, m, h \models Proven(A)$,
and by \eqref{A9} we know that $\mathcal{M}, m, h \models
\bigwedge_{i \in Ag}\neg Prove(i,A)$, thus also $\mathcal{M}, m, h
\models \langle j \rangle\bigwedge_{i \in Ag}\neg Prove(i,A)$ by
S5 reasoning for $[j]$.

\emph{Case 2}. The negative condition for $Prove(j,A)$ holds at
$(m,h)$. Then, since we have $\mathcal{M}, m, h \models \neg
Prove(j, A)$, the positive condition for $Prove(j,A)$ at $(m,h)$
must fail. Therefore, we can choose a $g \in Choice^m_j(h)$ such
that for no $t \in Pol$ do we have both $t \in Act(m,g)$ and
$\mathcal{M}, m \models t\co A$. Note, further, that $g \in
Choice^m_i(g)$ for every $i \in Ag$, and therefore the positive
condition for every formula of the form $Prove(i, A)$ fails at
$(m,g)$. Therefore, we must have $\mathcal{M}, m, g \models
\bigwedge_{i \in Ag}\neg Prove(i,A)$, and, since $g \in
Choice^m_j(h)$, also  $\mathcal{M}, m, h \models \langle j
\rangle\bigwedge_{i \in Ag}\neg Prove(i,A)$ as desired.

It only remains to show that \eqref{R4} preserves validity over
normal jstit models. Assume that $KA \to (\neg Proven(B_1)
\vee\ldots \vee\neg Proven(B_n))$ is valid over normal jstit
models, and assume also that we have:
$$
\mathcal{M}, m, h \models KA \wedge (\bigvee_{j \in Ag}
Prove(j,B_1) \wedge\ldots \wedge \bigvee_{j \in Ag} Prove(j,B_n)).
$$
This means that we can choose $j_{B_1},\ldots, j_{B_n} \in Ag$ in
such a way that we end up having:
$$
\mathcal{M}, m, h \models KA \wedge Prove(j_{B_1},B_1)
\wedge\ldots \wedge Prove(j_{B_n},B_n).
$$
We can now re-use the manner of reasoning employed above for the
base case of \eqref{A12}. More precisely, since $\mathcal{M}, m, h
\models Prove(j_{B_1},B_1)$ then $m$ must have some
$<$-successors, otherwise $h$ would be the unique history through
$m$. Then, if there existed $t \in Pol$ such that both $t \in
Act(m,h)$ and $\mathcal{M}, m \models t\co B_1$, the negative
condition for $Prove(j_{B_1},B_1)$ at $(m,h)$ would be violated.
On the other hand, if there were no such $t$, then the positive
condition for $Prove(j_{B_1},B_1)$ at $(m,h)$ would be violated.

Since $m$ is not a $\leq$-maximal moment in $Tree$, then we can
choose an $m' \in Tree$ such that both $m' > m$ and $h \in
H_{m'}$. All the histories passing through $m'$ are pairwise
undivided at $m$, therefore, by the presenting a new proof makes
histories divide constraint we must have $Act(m, g) = Act(m, g')$
for any $g,g' \in H_{m'}$. We also know that, since
$$
\mathcal{M},m, h \models  Prove(j_{B_1},B_1) \wedge\ldots\wedge
Prove(j_{B_n},B_n),
$$
there must be $t_1,\ldots, t_n \in Pol$ such that $t_1,\ldots, t_n
\in Act(m,h)$ and $\mathcal{M},m \models t_i\co B_i$ for all $i$
such that $1\leq i \leq n$. Since $h \in H_{m'}$, this further
means that

\noindent$t_1,\ldots, t_n \in \bigcap_{g \in H_{m'}}Act(m, g)$. By
the expansion of presented proofs constraint, we may infer from
the latter that $t_1,\ldots, t_n \in \bigcap_{g \in H_{m'}}Act(m',
g)$. By the future always matters constraint, we know that, since
$m < m'$, then we must have $R(m,m')$, whence, given that
$\mathcal{M},m \models t_i\co B_i$ for all $i$ such that $1\leq i
\leq n$, we must also have $\mathcal{M},m' \models t_i\co B_i$ for
all such $i$. Summing this up with $t_1,\ldots, t_n \in \bigcap_{g
\in H_{m'}}Act(m', g)$, we get that
$$
\mathcal{M},m' \models  Proven(B_1) \wedge\ldots \wedge
Proven(B_n).
$$
Further, we know that $\mathcal{M}, m \models KA$, so that by
$R(m,m')$ and S4 properties of $K$ we must also have $\mathcal{M},
m' \models KA$. Thus we get that $KA \wedge Proven(B_1)
\wedge\ldots\wedge Prove(B_n)$ is satisfied at $m'$ which is in
contradiction with the assumed validity of

\noindent$KA \to (\neg Proven(B_1) \vee\ldots\vee\neg
Proven(B_n))$.
\end{proof}

We then define a \emph{proof} in the above-presented axiomatic
system as a finite sequence of formulas such that every formula in
it is either an axiom or is obtained from earlier elements of the
sequence by one of the inference rules. A proof is a proof of its
last formula. If an $A \in Form$ is provable in our system, we
will write $\vdash A$.

The presence in our system of the rules like \eqref{R2} and
especially \eqref{R4} complicates the issue of finding the right
notion of an inference from premises and the right format for
Deduction Theorem. Given that these problems lie beyond the scope
of the present paper, we will take a little detour and will base
our definition of consistency of a set of formulas upon the notion
of provable formula, rather than just saying that a set $\Gamma
\subseteq Form$ is inconsistent iff $\bot$ is derivable from
$\Gamma$. Moreover, due to the form of our main result we need to
relativize our notions to sets of proof variables occurring in a
given set of formulas.

More precisely, assume that $Z \subseteq PVar$. Then we can define
$Pol_Z$ and $Form_Z$ as the sets of proof polynomials (resp.
formulas) containing proof variables from $Z$ only. Note that this
imposes no restrictions on proof constants, so that the set of
closed proof polynomials is contained in $Pol_Z$ for every $Z
\subseteq PVar$. Now, for a given $Z \subseteq PVar$ we say that
$\Gamma \subseteq Form_Z$ is a \emph{set of formulas in} $Z$. We
say that $\Gamma$ is \emph{inconsistent} iff for some
$A_1,\ldots,A_n \in \Gamma$ we have $\vdash (A_1 \wedge\ldots
\wedge A_n) \to \bot$, and we say that $\Gamma$ is consistent iff
it is not inconsistent. $\Gamma$ is \emph{maxiconsistent in} $Z$
iff $\Gamma \subseteq Form_Z$ and no consistent subset of $Form_Z$
properly extends $\Gamma$.

Even with this slightly non-standard definition of inconsistency,
we can still do many familiar things, e.g. extend consistent sets
with new formulas and eventually make them maxiconsistent. More
precisely, the following lemma holds:

\begin{lemma}\label{elementaryconsistency}
Let $Z \subseteq PVar$, let $\Gamma \subseteq Form_Z$ be
consistent, and let $A, B \in Form_Z$. Then:
\begin{enumerate}
\item There exists a $\Delta \subseteq Form_Z$ such that $\Delta$
is maxiconsistent in $Z$ and $\Gamma \subseteq \Delta$.

\item If $\Gamma$ is maxiconsistent in $Z$, then exactly one
element of $\{A, \neg A \}$ is in $\Gamma$.

\item If $\Gamma$ is maxiconsistent in $Z$, then $A \vee B \in
\Gamma$ iff $(A \in \Gamma$ or $B \in \Gamma)$.

\item If $\Gamma$ is maxiconsistent in $Z$ and $A, (A \to B) \in
\Gamma$, then $B \in \Gamma$.

\item If $\Gamma$ is maxiconsistent in $Z$, then $A \wedge B \in
\Gamma$ iff $(A \in \Gamma$ and $B \in \Gamma)$.
\end{enumerate}
\end{lemma}
\begin{proof} (Part 1) Just as in the standard case, we enumerate the
elements of $Form_Z$ as $A_1,\ldots, A_n,\ldots$ and form the
sequence of sets $\Gamma_1,\ldots, \Gamma_n,\ldots,$ such that
$\Gamma_1 := \Gamma$ and for every natural $i \geq 1$:
\begin{align*}
    \Gamma_{i + 1} :=
    \left\{%
\begin{array}{ll}
    \Gamma_i, & \hbox{ if $\Gamma_i \cup \{ A_ i \}$ is inconsistent;} \\
    \Gamma_i \cup \{ A_ i \}, & \hbox{ otherwise.} \\
\end{array}%
\right.
\end{align*}
We now define $\Delta := \bigcup_{i \geq 1}\Gamma_i$. Of course,
we have $\Gamma \subseteq \Delta$, and, moreover, $\Delta$ is
maxiconsistent in $Z$. To see this, note that for every $i \geq 1$
the set $\Gamma_i$ is consistent by construction. Now, if $\Delta$
is inconsistent, then there must be a valid implication from a
finite conjunction of formulas in $\Delta$ to $\bot$. These
formulas must be mentioned in our numeration of $Form_Z$ so that
the valid implication in question can presented as $\vdash
(A_{i_1} \wedge\ldots \wedge A_{i_n}) \to \bot$ for appropriate
natural $i_1,\ldots, i_n$. Since all of $A_{i_1}, \ldots, A_{i_n}$
are in $\Delta$, we must have, by the construction of
$\Gamma_1,\ldots, \Gamma_n,\ldots,$ that $A_{i_1}, \ldots, A_{i_n}
\in \Gamma_{max(i_1,\ldots, i_n)}$. But then this latter set must
be inconsistent which contradicts our construction.

Further, if some consistent $\Xi \subseteq Form_Z$ is such that
$\Delta \subset \Xi$, then let $A_n \in \Xi \setminus \Delta$. We
must have then $\Gamma_n \cup \{ A_n \}$ inconsistent, but we also
have $\Gamma_n \cup \{ A_n \} \subseteq \Xi$, which implies
inconsistency of $\Xi$, in contradiction to our assumptions.
Therefore, $\Delta$ is not only consistent, but also
maxiconsistent in $Z$.

(Part 2) We cannot have both $A$ and $\neg A$ in $\Gamma$, since
we have, of course, $\vdash (A \wedge \neg A) \to \bot$. If, on
the other hand, neither $A$, nor $\neg A$ is in $\Gamma$, then
both $\Gamma \cup \{ A \}$ and $\Gamma \cup \{ \neg A \}$ must be
inconsistent, so that for some $B_1, \ldots, B_n \in \Gamma$ we
will have:
$$
\vdash (B_1\wedge \ldots\wedge B_n \wedge A) \to \bot,
$$
whereas for some $C_1, \ldots, C_k \in \Gamma$ we will have:
$$
\vdash (C_1\wedge \ldots\wedge C_k \wedge \neg A) \to \bot,
$$
whence we get, using \eqref{A0} and \eqref{R1}:
$$
\vdash (C_1\wedge \ldots\wedge C_k) \to A,
$$
and further:
$$
\vdash (B_1\wedge \ldots\wedge B_n \wedge C_1\wedge \ldots\wedge
C_k) \to \bot,
$$
so that $\Gamma$ turns out to be inconsistent, contrary to our
assumptions.

(Part 3) Assume $(A \vee B) \in \Gamma$. If neither $A$ nor $B$
are in $\Gamma$, then, by Part 2, both $\neg A$ and $\neg B$ are
in $\Gamma$. Using \eqref{A0} and \eqref{R1} we get that:
$$
\vdash ((A \vee B) \wedge \neg A \wedge \neg B) \to \bot,
$$
showing that $\Gamma$ is inconsistent, contrary to our
assumptions. In the other direction, if, say $A \in \Gamma$ and
$(A \vee B) \notin \Gamma$, then, by Part 2, we must have $\neg(A
\vee B) \in \Gamma$. Using \eqref{A0} and \eqref{R1} we get that:
$$
\vdash (\neg(A \vee B) \wedge A) \to \bot,
$$
showing, again, that $\Gamma$ is inconsistent, contrary to our
assumptions. The case when $B \in \Gamma$ is similar.

Parts 4 and 5 are similar to Part 3.
\end{proof}

\textbf{Remark}. Note that one can recover the notion of
non-relativized maxiconsistent set and its properties just by
setting $Z : = PVar$. But this will not be needed in the present
paper.

We are now prepared to formulate our main result:

\begin{theorem}\label{completeness}
Let $X \subseteq PVar$ be such that $PVar \setminus X$ is
countably infinite. Then an arbitrary $\Gamma \subseteq Form_X$ is
consistent iff it is satisfiable in a normal jstit model.
\end{theorem}

The rest of the paper is mainly concerned with proving Theorem
\ref{completeness}. One part of it we have, of course, right away,
as a consequence of Theorem \ref{soundness}:

\begin{corollary}\label{c-soundness}
Let $X \subseteq PVar$ be such that $PVar \setminus X$ is
countably infinite. If $\Gamma \subseteq Form_X$ is satisfiable in
a normal jstit model, then $\Gamma$ is consistent.
\end{corollary}
\begin{proof}
Let $\Gamma \subseteq Form_X$ be satisfiable in a normal jstit
model so that we have, say $\mathcal{M}, m, h \models \Gamma$ for
some $(m,h) \in MH(\mathcal{M})$. If $\Gamma$ were inconsistent
this would mean that for some $A_1,\ldots,A_n \in \Gamma$ we would
have $\vdash (A_1 \wedge\ldots \wedge A_n) \to \bot$. By Theorem
\ref{soundness}, this would mean that:
$$
\mathcal{M}, m, h \models (A_1 \wedge\ldots \wedge A_n) \to \bot,
$$
whence clearly $\mathcal{M}, m, h \models \bot$, which is
impossible. Therefore, $\Gamma$ must be consistent.
\end{proof}

Before we move further, we mention some theorems in the above
axiom system to be used later in the proof of the main result:
\begin{lemma}\label{theorems}
The following holds for every $A \in Form$, $t \in Pol$, $x \in
PVar$, and $j \in Ag$:
\begin{enumerate}
\item $\vdash t\co A \to \Box t\co A$;

\item $\vdash Proven(A) \to \Box Proven(A)$;

\item $\not\vdash x\co A$;

\item $\vdash (Prove(j, A) \wedge \neg\Box Prove(j,A)) \to
[j](Prove(j, A) \wedge \neg\Box Prove(j,A))$;

\item $\vdash KA \to \Box KA$.
\end{enumerate}
\end{lemma}
\begin{proof}
(Part 1) We have:
\begin{align*}
    t\co A &\to !t\co t\co A&&\textup{(by \eqref{A5})}\\
    &\to Kt\co A&&\textup{(by \eqref{A5})}\\
    &\to \Box K\Box t\co A&&\textup{(by \eqref{A8})}\\
    &\to K\Box t\co A&&\textup{(by \eqref{A1})}\\
    &\to \Box t\co A&&\textup{(by \eqref{A7})}
\end{align*}
 Our theorem follows then by transitivity of implication.

(Part 2) Again, we proceed by building a chain of implications:
\begin{align*}
    Proven(A) &\to KProven(A)&&\textup{(by \eqref{A11})}\\
    &\to \Box K\Box Proven(A)&&\textup{(by \eqref{A8})}\\
    &\to K\Box Proven(A)&&\textup{(by \eqref{A1})}\\
    &\to \Box Proven(A)&&\textup{(by \eqref{A7})}
\end{align*}

(Part 3). Take an arbitrary normal jstit model $\mathcal{M} =
\langle Tree, \leq, Choice, Act, R, \mathcal{E}, V\rangle$ and
consider another model $\mathcal{M}' = \langle Tree, \leq, Choice,
Act, R, \mathcal{E}', V\rangle$ such that:
\begin{align*}
    \mathcal{E}'(m,t) = \left\{%
\begin{array}{ll}
    \mathcal{E}(m,t), & \hbox{if $t \neq x$;} \\
    \emptyset, & \hbox{if $t = x$.} \\
\end{array}%
\right.
\end{align*}
It is straightforward to verify that  $\mathcal{M}'$ is again a
normal jstit model, and we obviously have $\mathcal{M}', m
\not\models x\co A$ for every $m \in Tree$. Therefore, $x\co A$ is
not valid, and, by Theorem \ref{soundness}, cannot be provable in
our system.

(Part 4). We chain the implications as follows:
\begin{align*}
    (Prove(j, A) \wedge \neg\Box Prove(j,A)) &\to ([j]Prove(j, A) \wedge \Box\neg\Box Prove(j,A))&&\textup{(by \eqref{A1} and \eqref{A9})}\\
    &\to ([j]Prove(j, A) \wedge [j]\neg\Box Prove(j,A))&&\textup{(by \eqref{A2})}\\
    &\to [j](Prove(j, A) \wedge \neg\Box Prove(j,A))&&\textup{(by \eqref{A1})}
\end{align*}
\end{proof}

(Part 5). By S5 properties of $\Box$ and S4 properties of $K$, we
clearly have $\vdash \Box K\Box A \to \Box KA$. Part 5 follows
then by \eqref{A8} and transitivity of implication.

\section{The canonical model}\label{canonicalmodel}

The main aim of the present section is to prove the inverse of
Corollary \ref{c-soundness}. The method used is a variant of the
canonical model technique, but, due to the complexity of the case,
we do not define our model in one full sweep. Rather, we proceed
piecewise, defining elements of the model one by one, and checking
the relevant constraints as soon, as we have got enough parts of
the model in place. The last subsection proves the truth lemma for
the defined model.

Throughout this section we fix an $X \subseteq PVar$ such that
$PVar \setminus X$ is countably infinite. We then
present\footnote{More precisely, we divide $PVar \setminus X$ into
three countably infinite subsets plus a single proof variable
which we will denote $z$. For the first of these three subsets
(denoted $Y$) we fix a bijection onto the Cartesian product of
$Ag$ and $Form$, for the other two (denoted $W$ and $U$) we fix
bijections onto $Form$.} the set of proof variables in the
following form:
$$
PVar = X \cup Y \cup W \cup U \cup \{ z \},
$$
where:
$$
Y := \{ y_{(i,A)} \mid i \in Ag, A \in Form_X \},
$$
$$
W := \{ w_A \mid A \in Form_X \},
$$
$$
U := \{ u_A \mid A \in Form_X \}.
$$
Since $Form$ is countably infinite and $Ag$ is finite, this
presentation of $PVar$ is well-defined. Also throughout this
section we will use $\mathcal{M} = \langle Tree, \leq, Choice,
Act, R, \mathcal{E}, V\rangle$ as a fixed notation for our
canonical model.

The ultimate building blocks of $\mathcal{M}$ we will call
\emph{elements}. Before going on with the definition of
$\mathcal{M}$, we define what these elements are and explore some
of their properties.

\begin{definition}\label{element}
An element is a sequence of the form $(\Gamma_1,\ldots,\Gamma_n)
\alpha$ for some natural $n \geq 1$ such that:
\begin{itemize}
\item $\alpha \in \{ \uparrow, \downarrow \}$;

\item For every $i \leq n$, $\Gamma_i$ is maxiconsistent in $X$;

\item For every $i < n$, for all $A \in Form_X$, if $KA \in
\Gamma_i$, then $KA \in \Gamma_{i + 1}$;

\item For every $i$ such that $1 < i \leq n$, for all $j \in Ag$
and $A \in Form_X$, if $Prove(j,A) \in \Gamma_1$, then $Proven(A)
\in \Gamma_i$;

\item For every $i$ such that $1 < i \leq n$, for all $A \in
Form_X$, it is true that

\noindent$K\Box\neg Prove(j, A) \in \Gamma_i$.
\end{itemize}
\end{definition}
In other words, elements are sequences of subsets of $Form_X$ of a
rather special kind, which are signed by either $\downarrow$ or
$\uparrow$. The (purely technical) reason for including these
arrows in the structure of elements is that one normally needs at
least two copies of one element in order to get the truth
conditions for formulas like $\Box Prove(j,A)$ right. Both
$\downarrow$ or $\uparrow$ mainly become relevant after we define
$Act$ and for most other purposes they can be more or less
overlooked.

We prove the following lemma:
\begin{lemma}\label{elementcontinuation}
Whenever $(\Gamma_1,\ldots,\Gamma_n)\alpha$ is an element, then,
for some $\Gamma_{n + 1} \subseteq Form_X$, the sequence
$(\Gamma_1,\ldots,\Gamma_{n + 1})\alpha$ is also an element.
\end{lemma}
\begin{proof}
Assume $(\Gamma_1,\ldots,\Gamma_n)\alpha$ is an element. We have
two cases to consider:

\emph{Case 1.} $n = 1$. Then consider the set:

\begin{align*}
\Delta := \{ KA \mid KA \in \Gamma_1 \} \cup \{ Proven(A) \mid
(\exists j \in Ag)&(Prove(j, A) \in \Gamma_1) \} \cup\\
&\cup \{ K\Box\neg Prove(j, A) \mid A \in Form \}.
\end{align*}

We show that $\Delta$ is consistent. Of course, the set $\{ KA
\mid KA \in \Gamma_1 \}$ is consistent since it is a subset of
$\Gamma_1$ and the latter is assumed to be consistent.

Further, if the set
$$
\Delta' := \{ KA \mid KA \in \Gamma_1 \} \cup \{ Proven(A) \mid
(\exists j \in Ag)(Prove(j, A) \in \Gamma_1) \}
$$
is inconsistent, this would mean, wlog, that for some $B_1,\ldots,
B_k, C_ 1,\ldots, C_l \in Form$ and $j_1, \ldots, j_l \in Ag$ such
that $KB_1,\ldots, KB_k$ and $Prove(j_1, C_1),\ldots,
Prove(j_l,C_l)$ are in $\Gamma_1$, we have that:
$$
\vdash(KB_1\wedge\ldots \wedge KB_k) \to (\neg Proven(C_1)
\vee\ldots \vee \neg Proven(C_l)),
$$
whence, by \eqref{A7}:
$$
\vdash K(B_1\wedge\ldots \wedge B_k) \to (\neg Proven(C_1)
\vee\ldots \vee \neg Proven(C_l)),
$$
and further, by \eqref{R4}:
$$
\vdash K(B_1\wedge\ldots \wedge B_k) \to (\bigwedge_{j \in Ag}\neg
Prove(j, C_1) \vee\ldots \vee \bigwedge_{j \in Ag}\neg
Prove(j,C_l)).
$$
Since the latter formula is in $X$ it is, of course, in $\Gamma_1$
by its maxiconsistency in $X$. Also, given Lemma
\ref{elementaryconsistency}, $K(B_1\wedge\ldots \wedge B_k)$ is in
$\Gamma_1$ by the fact that $KB_1,\ldots, KB_k \in \Gamma_1$,
\eqref{A7}, and the fact that $\Gamma_1$ is maxiconsistent in $X$.
Therefore, we get:
$$
\bigwedge_{j \in Ag}\neg Prove(j, C_1) \vee\ldots \vee
\bigwedge_{j \in Ag}\neg Prove(j,C_l) \in \Gamma_1
$$
by Lemma \ref{elementaryconsistency}.4. By Lemma
\ref{elementaryconsistency}.3, we further get that for some $r$
such that $1 \leq r \leq l$ all of the formulas $\neg Prove(j,
C_r)$, where $j \in Ag$ are in $\Gamma_1$. By the choice of
$C_1,\ldots, C_l$ this makes $\Gamma_1$ inconsistent and we get a
contradiction, which shows that $\Delta'$ is consistent.

Assume, further, that $\Delta$ is inconsistent. In view of
consistency of $\Delta'$ this will mean that for some
$KB_1,\ldots, KB_k$ in $\Gamma_1$, and some $Proven(C_1),\ldots,
Proven(C_l)$ from $\Delta' \setminus \Gamma_1$ and some
$Prove(j_1,D_1),\ldots, Prove(j_r,D_r) \in Form$, we will have:
\begin{align*}
\vdash(KB_1\wedge\ldots \wedge KB_k \wedge &Proven(C_1)
\wedge\ldots
\wedge Proven(C_l)) \to\\
&\to (\langle K \rangle\Diamond Prove(j_1,D_1)\vee\ldots
\vee\langle K \rangle\Diamond Prove(j_r,D_r)).
\end{align*}
From the latter validity, by \eqref{R2} and \eqref{A7} we get
that:
\begin{align*}
\vdash K(KB_1\wedge\ldots \wedge KB_k \wedge &Proven(C_1)
\wedge\ldots
\wedge Proven(C_l)) \to\\
&\to K(\langle K \rangle\Diamond Prove(j_1,D_1)\vee\ldots
\vee\langle K \rangle\Diamond Prove(j_r,D_r)),
\end{align*}
whence, by \eqref{A12}, we obtain:
\begin{align*}
\vdash K(KB_1\wedge\ldots \wedge KB_k \wedge Proven(C_1)
\wedge\ldots \wedge Proven(C_l)) \to \bot,
\end{align*}
and, by \eqref{A7}, and \eqref{A11} we further obtain:
\begin{align*}
\vdash (KB_1\wedge\ldots \wedge KB_k \wedge Proven(C_1)
\wedge\ldots \wedge Proven(C_l)) \to \bot,
\end{align*}
showing that $\Delta'$ must be inconsistent. This makes up a
contradiction showing that $\Delta$ must be consistent.

Since $\Delta$ is shown to be consistent, then, by Lemma
\ref{elementaryconsistency}.1, it is also extendable to a set
$\Gamma_2$ which is maxiconsistent in $X$. By the choice of
$\Delta$, this means that $(\Gamma_1, \Gamma_2)\alpha$ must be an
element.

\emph{Case 2}. $n > 1$. Then it is easy to see that
$(\Gamma_1,\ldots, \Gamma_n, \Gamma_n)\alpha$ is an element.
\end{proof}
The structure of elements will be important in what follows. If
$\xi = (\Gamma_1,\ldots, \Gamma_n)\alpha$ is an element, then its
\emph{initial segment} is any element $\tau$ of the form
$(\Gamma_1,\ldots, \Gamma_k)\alpha$ with $k \leq n$. If, moreover,
$k < n$, then $\tau$ is a \emph{proper} initial segment of $\xi$,
and if $k = n -1$, then $\tau$ is the \emph{greatest} proper
initial segment of $\xi$. Moreover, we define $n$ to be the
\emph{length} of $\xi$. Thus, any element of length $1$ has no
proper initial segments. Furthermore, we define that $\Gamma_n$ is
the end element of $\xi$ and write $\Gamma_n = end(\xi)$.

We now define the canonical model using elements as our building
blocks. We start by defining the following relation $\equiv$
between elements of equal length. For the elements of length $1$
we set:
$$
(\Gamma)\alpha \equiv (\Delta)\beta \Leftrightarrow (\forall A \in
Form_X)(\Box A \in \Gamma \Rightarrow A \in \Delta);
$$
and for the elements of greater length we set:

\begin{align*}
(\Gamma_1,\ldots, \Gamma_n, \Gamma_{n + 1})\alpha \equiv
&(\Delta_1,\ldots, \Delta_n, \Delta_{n + 1})\beta \Leftrightarrow\\
&\Leftrightarrow (\Gamma_1 = \Delta_1 \wedge \ldots \wedge
\Gamma_n = \Delta_n \wedge \alpha = \beta \wedge (\Gamma_{n +
1})\alpha \equiv (\Delta_{n + 1})\beta).
\end{align*}

It is routine to check that $\equiv$ is an equivalence relation
given that $\Box$ is an S5 modality. We will denote the
equivalence class of element $(\Gamma_1,\ldots, \Gamma_n)\alpha$
generated by $\equiv$ by $[(\Gamma_1,\ldots,
\Gamma_n)\alpha]_\equiv$. Since all the elements inside a given
$\equiv$-equvalence class are of the same length, we may extend
the notion of length to these classes setting that the length of
$[(\Gamma_1,\ldots, \Gamma_n)\alpha]_\equiv$ also equals $n$.

We now proceed to definitions of components for the canonical
model.

\subsection{$Tree$, $\leq$, and $Hist$}

The first two components of the canonical model $\mathcal{M}$ are
as follows:

\begin{itemize}
    \item $Tree$ is the set of $\equiv$-equvalence classes of
    elements plus $\dag$ and $\ddag$ as additional moments;
    \item We set that both $\dag < m$ and $m \not< \dag$ for every
     $m \in Tree \setminus \{ \dag \}$. We further set that $\ddag$ is only
    $<$-comparable to $\dag$ (in which case we already have $\dag < \ddag$), and for any two $\equiv$-equvalence classes of
    elements $m$ and $m'$, we have that $m < m'$ iff
    there is an element $\xi \in
    m$ such that $\xi$ is a proper initial segment of every element $\tau \in
    m'$. The relation $\leq$ is then defined as the reflexive
    companion to $<$.
\end{itemize}

Before we move on to the choice- and justifications-related
components, let us pause to check that the restraints imposed by
our semantics on $Tree$ and $\leq$ are satisfied:
\begin{lemma}\label{leq}
The relation $\leq$, as defined above, is a partial order on
$Tree$, which satisfies both historical connection and no backward
branching constraints. Moreover, every element in $Tree$, except
for $\ddag$, has at least one immediate $<$-successor.
\end{lemma}
\begin{proof}
Transitivity and reflexivity of $\leq$ are obvious. As for
antisymmetry, assume that we have both $m < m'$ and $m' < m$. Then
$m$ and $m'$ must be equivalence classes of elements. Let $\xi \in
m$ be a proper initial segment of every element in $m'$ and let
$\tau \in m'$ be a proper initial segment of every element in $m$.
It follows that $\xi$ is a proper initial segment of $\tau$ and
also $\tau$ is a proper initial segment of $\xi$, a contradiction.

Historical connection is satisfied since $\dag$ is the
$\leq$-least element of $Tree$. Let us prove the absence of
backward branching. Assume that we have both $m \leq m''$ and $m'
\leq m''$ but neither $m \leq m'$ nor $m' \leq m$ holds. This
means that all the three moments are pairwise different and none
of them is either $\dag$ or $\ddag$, otherwise our assumptions
about them would be immediately falsified. Therefore, all the
three moments are some equivalence classes of elements and we also
have $m \neq m'$, $m < m''$ and $m' < m''$. So let $\xi \in m$ and
$\tau \in m'$ be such that both $\xi$ and $\tau$ are proper
initial segments of every element in $m''$. Then, since $m \neq
m'$, $\xi$ and $\tau$ must be different, hence either $\xi$ must
be a proper initial segment of $\tau$ or $\tau$ must be a proper
initial segment of $\xi$. Assume, wlog, that $\xi$ is a proper
initial segment of $\tau$. Then $\xi$ is included into the
greatest proper initial segment of $\tau$. Let $\tau'$ be any
element in $m'$. It follows from the definition of $\equiv$ that
all the elements within $m'$ share the same greatest proper
initial segment, therefore $\xi$ must be a proper initial segment
of $\tau'$ as well. It follows that $m < m'$, contrary to our
assumptions.

Consider the $<$-successors of a given $m \in Tree$. If $m \neq
\ddag$, then either $m = \dag$, or $m$ is an equivalence class of
elements. If $m = \dag$, then take any $\Gamma \subseteq Form_X$
which is maxiconsistent in $X$ and any $\alpha \in \{ \uparrow,
\downarrow \}$. Then $(\Gamma)\alpha$ is an element and we have
$\dag < [(\Gamma)\alpha]_\equiv$. Moreover, no other moment can be
in between them: this cannot be either $\dag$, or $\ddag$, or an
equivalence class of elements (since the greatest proper initial
segment of $(\Gamma)\alpha$ is empty). If, on the other hand, $m$
is an equivalence class of elements, then assume that $m =
[(\Gamma_1,\ldots, \Gamma_n)\alpha]_\equiv$. By Lemma
\ref{elementcontinuation}, we know that for some $\Delta \subseteq
Form$, the tuple $(\Gamma_1,\ldots, \Gamma_n, \Delta)\alpha$ must
be an element. But then we must have
$$
[(\Gamma_1,\ldots, \Gamma_n)\alpha]_\equiv < [(\Gamma_1,\ldots,
\Gamma_n, \Delta)\alpha]_\equiv,
$$
and again, no moments are strictly in between them since
$(\Gamma_1,\ldots, \Gamma_n)\alpha$ is the greatest proper initial
segment of $(\Gamma_1,\ldots, \Gamma_n, \Delta)\alpha$.
\end{proof}

Moreover, it is easy to see that if $m, m' \in Tree$ are two
equivalence classes of elements, and $m < m'$, then the length of
$m$ is less than the length of $m'$, and if $m'$ is an immediate
$<$-successor of $m$, then length of $m$ is the length of $m'$
minus one.

Before we move on, let us have a quick look into the structure of
histories induced by $Tree$ and $\leq$. Lemma \ref{leq} shows that
we must have the history $(\dag,\ddag)$ plus a bunch of infinite
histories of the form $(\dag, m_1,\ldots, m_n,\ldots)$, ordered in
the type of $\omega$, where, for $n \geq 1$, $m_n$ is an
equivalence class of elements of length $n$ and every next element
is the immediate $<$-successor of the previous one. Every such
infinite history we can also represent in the form $(\dag,
\xi_1,\ldots, \xi_n,\ldots)$ such that for every $n \geq 1$:
\begin{itemize}
    \item $\xi_n \in m_n$;
    \item $\xi_n$ is the greatest proper initial segment of
    $\xi_{n + 1}$.
\end{itemize}
Moreover, one can show that such a representation, for a given
history of the form $(\dag, m_1,\ldots, m_n,\ldots)$, is unique.
Indeed, suppose that $(\dag, \xi_1,\ldots, \xi_n,\ldots)$ and
$(\dag, \xi'_1,\ldots, \xi'_n,\ldots)$ are two different
representations for $(\dag, m_1,\ldots, m_n,\ldots)$. Then let
$i\geq 1$ be the first natural number such that $\xi_i \neq
\xi'_i$. Consider $m_{i + 1}$. We have  $\xi_{i + 1}, \xi'_{i + 1}
\in m_{i + 1}$ so that $\xi_{i + 1} \equiv \xi'_{i + 1}$. Since
$\xi_i$ and $\xi'_i$ are the greatest proper initial segments of
$\xi_{i + 1}$, $\xi'_{i + 1}$, respectively, the greatest proper
initial segments of $\xi_{i + 1}$ and $\xi'_{i + 1}$ are non-empty
and, by $\xi_{i + 1} \equiv \xi'_{i + 1}$, must coincide, which
cannot be the case since $\xi_i \neq \xi'_i$.

Therefore, if $h = (\dag, m_1,\ldots, m_n,\ldots)$ is a history in
$Tree$ and $(\dag, \xi_1,\ldots, \xi_n,\ldots)$ is the unique
representation of $h$ as a sequence of elements, we define $m_n
\cap h$ to be $\xi_n$.

All the above statements admit of an inversion. Not only can every
history be uniquely represented as a sequence of elements, but
also every sequence of elements of an appropriate form represents
a unique history in $\mathcal{M}$. Not only is every intersection
of an equivalence class of elements and a history an element in
this class, but also, conversely, every element defines the
intersection of at least one history with the equivalence class
induced by this element. More precisely:
\begin{lemma}\label{hist}
The following statements are true:
\begin{enumerate}
\item Fix a sequence $(\dag, \xi_1,\ldots, \xi_n,\ldots)$ where
all of $\xi_1,\ldots, \xi_n,\ldots $ are elements and, for every
natural $n$, $\xi_n$ is the greatest proper initial segment of
    $\xi_{n + 1}$. Then there is a unique history $h = (\dag, m_1,\ldots,
    m_n,\ldots)$ in $\mathcal{M}$ such that for all natural $n$ it
    is true that $\xi_n \in m_n$ (thus $\xi_n = m_n \cap h$).

\item Let $\xi$ be an element. Then there is at least one history
$h \in H_{[\xi]_\equiv}$ such that $[\xi]_\equiv \cap h = \xi$.
\end{enumerate}
\end{lemma}
\begin{proof}
As for Part 1, consider $(\dag, [\xi_1]_\equiv,\ldots,
[\xi_n]_\equiv,\ldots)$; it is obviously a history in
$\mathcal{M}$ and we also have $\xi_n \in [\xi_n]_\equiv$ for all
natural $n$.

As for Part 2, we have to consider two cases.

\emph{Case 1}. The length of $\xi$ equals $1$, so that $\xi =
(\Gamma)\alpha$ for appropriate $\Gamma$ and $\alpha$. Then we
know, by the proof of Lemma \ref{elementcontinuation} above, that
for some $\Delta \subseteq Form_X$ the sequence:
\begin{align*}
    \xi_1 &:= \xi = (\Gamma)\alpha;\\
    \xi_2 &:= (\Gamma, \Delta)\alpha;\\
    &\ldots;\\
    \xi_{n + 1} &:= (\Gamma, \underbrace{\Delta, \ldots, \Delta}_{n\textup{ times}})\alpha;\\
    &\ldots;
\end{align*}
is a sequence of elements in which every element is the greatest
proper initial segment of the next one. Therefore, by Part 1,
there must be a history $h$ in $\mathcal{M}$ such that $h = (\dag,
[\xi_1]_\equiv,\ldots, [\xi_n]_\equiv,\ldots)$ and also $\xi =
\xi_1 = [\xi_1]_\equiv \cap h$.

\emph{Case 2}. The length of $\xi$ is greater than $1$, so that
$\xi = (\Gamma_1,\ldots, \Gamma_n)\alpha$ for appropriate $n > 1$,
$\Gamma_1,\ldots, \Gamma_n$ and $\alpha$. Then we define the
following sequence of elements:
\begin{align*}
    \xi_1 &:= (\Gamma_1)\alpha;\\
    \xi_2 &:= (\Gamma_1, \Gamma_2)\alpha;\\
    &\ldots;\\
    \xi_n &:= \xi = (\Gamma_1,\ldots, \Gamma_n)\alpha;\\
    \xi_{n + 1} &:= (\Gamma_1,\ldots, \Gamma_n, \Gamma_n)\alpha;\\
    &\ldots;\\
    \xi_{n + k} &:= (\Gamma_1,\ldots, \Gamma_n, \underbrace{\Gamma_n, \ldots, \Gamma_n}_{k\textup{ times}})\alpha;\\
    &\ldots.
\end{align*}
Again, it is easy to see that every element in this sequence is
the greatest proper initial segment of the next one, so that,
arguing as in the previous case, we get that  $h = (\dag,
[\xi_1]_\equiv,\ldots, [\xi_n]_\equiv,\ldots, [\xi_{n +
k}]_\equiv,\ldots)$ is a history in $\mathcal{M}$ and $\xi = \xi_n
= [\xi_n]_\equiv \cap h$.
\end{proof}

\subsection{$Choice$}

We now define the choice structures of our canonical model:
\begin{itemize}
    \item $Choice^m_j = H_m$, if $m \in \{ \dag, \ddag \}$;
    \item $Choice^m_j(h) = \{ h' \mid h' \in H_m,\,
    (\forall A \in Form)([j]A \in end(h \cap m) \Rightarrow A \in end(h' \cap
    m))\}$, if $m$ is an equivalence class of elements.
\end{itemize}
Since for every $j \in Ag$, $[j]$ is an S5-modality, $Choice$
induces a partition on $H_m$ for every given $m \in Tree$. We
check that the choice function verifies the relevant semantic
constraints:
\begin{lemma}\label{choice}
The tuple $\langle Tree, \leq, Choice\rangle$, as defined above,
verifies both the independence of agents and the no choice between
undivided histories constraints.
\end{lemma}
\begin{proof}
We first tackle no choice between undivided histories. Consider a
moment $m$ and two histories $h, h' \in H_m$ such that $h$ and
$h'$ are undivided at $m$. Since the agents' choices are only
non-vacuous at moments represented by equivalence classes  of
elements, we may safely assume that $m$ is such a class. Since $h$
and $h'$ are undivided at $m$, this means that there is a moment
$m'$ such that $m < m'$ and $m'$ is shared by $h$ and $h'$. Hence
we know that also $m'$ is some equivalence class of elements.
Suppose the length of $m$ is $n$ and the length of $m'$ is $n'$.
Then $n < n'$, also $h \cap m$ is the initial segment of length
$n$ of $h \cap m'$, and similarly, $h' \cap m$ is the initial
segment of length $n$ of $h' \cap m'$. But both $h \cap m'$ and
$h' \cap m'$ are, by definition, in $m'$, therefore, they must
share the greatest proper initial segment. Hence, their initial
segments of length $n$ must coincide as well, and we must have $h
\cap m = h' \cap m$, whence $end(h \cap m) = end(h' \cap m)$. Now,
if $j \in Ag$ and $[j]A \in end(h \cap m)$, then, by \eqref{A1}
and maxiconsistency of $end(h \cap m)$ in $X$, we will have also
$A \in end(h \cap m) = end(h' \cap m)$, and thus $h' \in
Choice^m_j(h)$, so that $Choice^m_j(h) = Choice^m_j(h')$ since
$Choice$ is a partition of $H_m$.

Consider, next, the independence of agents. Let $m \in Tree$ and
let $f$ be a function on $Ag$ such that $\forall j \in Ag(f(j) \in
Choice^m_j)$. We are going to show that in this case $\bigcap_{j
\in Ag}f(j) \neq \emptyset$. If $m \in \{ \dag, \ddag \}$, then
this is obvious, since every agent will have a vacuous choice. We
treat the case when $m$ is an equivalence class of elements.
Assume that $m = [(\Gamma_1,\ldots, \Gamma_{n +
1})\alpha]_\equiv$. We have two cases to consider:

\emph{Case 1}. $n = 0$. By \eqref{A1} we know that there is a set
$\Delta$ of formulas of the form $\Box A$ which is shared by all
sets of the form $end(\xi)$ with $\xi \in m$ in the sense that if
$\xi \in m$, then $\Box A \in end(\xi)$ iff $\Box A \in \Delta$.
By the same axiom scheme and Lemma \ref{hist}.2, we also know that
for every $j \in Ag$ there is set $\Delta_j$ of formulas of the
form $[j]A$ which is shared by all sets of the form $end(\xi)$
such that $\exists h(h \in f(j) \wedge \xi = m \cap h)$. More
precisely:
$$
\xi \in m \Rightarrow (\exists h(h \in f(j) \wedge \xi = m \cap h)
\Leftrightarrow (\forall A \in Form)([j]A \in end(\xi)
\Leftrightarrow [j]A \in \Delta_j)).
$$

We now consider the set $\Delta \cup \bigcup\{ \Delta_j\mid j \in
Ag \}$ and show its consistency. Indeed, if this set is
inconsistent, then, wlog, we would have a provable formula of the
following form:

\begin{equation}\label{E:c1}
\vdash (\Box A \wedge \bigwedge_{j \in Ag}[j]A_j) \to \bot.
\end{equation}

But then, choose for every $j \in Ag$ an element $\xi_j \in m$
such that

\noindent$(\forall A \in Form)([j]A \in end(\xi_j) \Leftrightarrow
[j]A \in \Delta_j)$. This is possible, since we may simply choose
an arbitrary $h_j \in f(j)$ and set $\xi_j: = m \cap h_j$. Then we
will have $[j]A_j \in \xi_j$ for every $j \in Ag$. Next, consider
$\Gamma_1$. Since $m = [(\Gamma_1)\alpha]_\equiv$ and $\Box$ is an
S5-modality, we must have:
$$
\{ \Diamond[j]A_j \in Ag \} \subseteq \Gamma_1,
$$
whence, by Lemma \ref{elementaryconsistency}.5:
$$
\bigwedge_{j \in Ag}\Diamond[j]A_j \in \Gamma_1,
$$
and further, by \eqref{A3} and Lemma
\ref{elementaryconsistency}.4:
$$
\Diamond\bigwedge_{j \in Ag}[j]A_j \in \Gamma_1.
$$
Also, by definition of $\Delta$ and the fact that $m =
[(\Gamma_1)\alpha]_\equiv$, we get successively:
$$
\Box A \in  \Gamma_1,
$$
then, by Lemma \ref{elementaryconsistency}.5:
$$
\Box A \wedge \Diamond\bigwedge_{j \in Ag}[j]A_j \in \Gamma_1,
$$
and finally, by the fact that $\Box$ is an S5-modality:

\begin{equation}\label{E:c2}
\Diamond(\Box A \wedge \bigwedge_{j \in Ag}[j]A_j) \in \Gamma_1.
\end{equation}

From \eqref{E:c1}, together with \eqref{E:c2}, it follows by S5
reasoning for $\Box$ that $\Diamond\bot \in \Gamma_1$, so that,
again by S5 properties of $\Box$ and Lemma
\ref{elementaryconsistency}.4, it follows that $\bot \in
\Gamma_1$, which is in contradiction with maxiconsistency of
$\Gamma_1$ in $X$.

Hence $\Delta \cup \bigcup\{ \Delta_j\mid j \in Ag \}$ is
consistent, and since it is in $X$, we can extend it to a set
$\Xi$ which is maxiconsistent in $X$. We now consider
$(\Xi)\alpha$ which is obviously an element, and since, moreover
$\Delta \subseteq \Xi$, then also $(\Xi)\alpha \in m$. By Lemma
\ref{hist}.2, we can choose a history $g$ such that $(\Xi)\alpha =
g \cap m$. We also know that for every $j \in Ag$, there is a
history $h_j \in f(j)$ such that $h_j \cap m = \xi_j$ by the
choice of $\xi_j$. Therefore, for every $j \in Ag$,
$Choice^m_j(h_j) = f(j)$. Also, if $[j]A \in end(\xi_j) = end(h_j
\cap m)$, then $[j]A \in \Delta_j$, hence $[j]A \in \Xi = end(g
\cap m)$, therefore, by \eqref{A1}, $A \in end(g \cap m)$. Thus we
get that $g \in \bigcap_{j \in Ag}Choice^m_j(h_j) = \bigcap_{j \in
Ag}f_j$ so that the independence of agents is verified.

\emph{Case 2}. $n > 0$. For the most part, we can re-use our
reasoning from Case 1. We again form the sets $\Delta$, $\{
\Delta_j\mid j \in Ag \}$ and $\Xi$, and consider element
$(\Gamma_1,\ldots, \Gamma_n, \Xi)\alpha \in m$. We then choose a
history $g \in H_m$ for which we have $(\Gamma_1,\ldots, \Gamma_n,
\Xi)\alpha = m \cap g$ and show that $g \in \bigcap_{j \in
Ag}Choice^m_j(h_j) = \bigcap_{j \in Ag}f_j$.

The only new ingredient is that now seeing that $(\Gamma_1,\ldots,
\Gamma_n, \Xi)\alpha$ is in fact an element is much less trivial
and has to be argued separately. We show this as follows. If $KA
\in \Gamma_n$, then $KA \in \Gamma_{n + 1}$ by definition of an
element. But then $\Box KA \in \Gamma_{n + 1}$ by Lemma
\ref{theorems}.5 and maxiconsistency of $\Gamma_{n + 1}$ in $X$,
whence $\Box KA \in \Delta$ and, therefore, $\Box KA \in \Xi$. By
\eqref{A1} and maxiconsistency of $\Xi$ we get then $KA \in \Xi$.
Similarly, if $Prove(j,A) \in \Gamma_1$, then $Proven(A) \in
\Gamma_{n + 1}$ by definition of an element. But then $\Box
Proven(A) \in \Gamma_{n + 1}$, by Lemma \ref{theorems}.2 and
maxiconsistency of $\Gamma_{n + 1}$  in $X$, whence $\Box
Proven(A) \in \Delta$ and, therefore, $\Box Proven(A) \in \Xi$. By
\eqref{A1} and maxiconsistency of $\Xi$ in $X$, we get then
$Proven(A) \in \Xi$. Finally, if $A \in Form_X$ and $j \in Ag$,
then, since $n > 0$, we must have $K\Box\neg Prove(j,A) \in
\Gamma_{n + 1}$ by definition of an element, whence $\Box
K\Box\neg Prove(j,A) \in \Gamma_{n + 1}$ by \eqref{A8} and
maxiconsistency of $\Gamma_{n + 1}$  in $X$, so that $\Box
K\Box\neg Prove(j,A) \in \Delta$ and, further, $\Box K\Box\neg
Prove(j,A) \in \Xi$. By \eqref{A1} and maxiconsistency of $\Xi$ in
$X$, we get then that $K\Box\neg Prove(j,A) \in \Xi$. Thus
$(\Gamma_1,\ldots, \Gamma_n, \Xi)\alpha$ is an element, and the
rest is shown exactly as in Case 1.
\end{proof}

\subsection{$R$ and $\mathcal{E}$}

We now define the justifications-related components of our
canonical model. We first define $R$ as follows:
\begin{itemize}
    \item $R([(\Gamma)\alpha]_\equiv, m')\Leftrightarrow (m' \in Tree \setminus \{ \dag,\ddag \})\&$

    $\qquad\qquad\qquad\qquad\qquad\quad\&(\forall
    \tau \in m')(\forall A \in Form_X)(KA \in \Gamma \Rightarrow KA \in
    end(\tau))$;
    \item If $n > 1$, then\begin{align*}
R([(\Gamma_1,\ldots, \Gamma_n)&\alpha]_\equiv, m')
\Leftrightarrow\\
&\Leftrightarrow (\exists \Delta_1,\ldots, \Delta_k \subseteq
Form_X)(k > 0 \& m' =
[(\Gamma_ 1, \Delta_1,\ldots, \Delta_k)\alpha]_\equiv \&\\
 &\qquad\qquad\qquad\qquad\qquad\&(\forall A \in Form_X)(KA \in \Gamma_n \Rightarrow KA \in
    \Delta_k));
\end{align*}
    \item $R(\dag,m)$, for all $m \in Tree$;
    \item $R(\ddag,m) \Leftrightarrow m = \ddag$.
\end{itemize}

Now, for the definition of $\mathcal{E}$:
\begin{itemize}
    \item For
    all $t \in Pol$: $\mathcal{E}(\dag, t) = \mathcal{E}(\ddag, t) = \{ A \in Form \mid
\vdash t\co A \}$;

\item For all $t \in Pol_X$ and $m \in Tree \setminus \{ \dag,
\ddag \}$:
\begin{align*}
    (\forall A \in Form)(&A \in \mathcal{E}(m, t)
    \Leftrightarrow\\
    &\Leftrightarrow (\exists
    t_1\co B_1)\ldots(\exists t_n\co B_n)((\forall \xi \in m)(t_1\co B_1,\ldots,t_n\co B_n \in
    end(\xi)) \&\\
    &\qquad\qquad\qquad\qquad\qquad\qquad\quad\& \vdash (t_1\co B_1\wedge\ldots \wedge t_n\co B_n) \to t\co A));
\end{align*}
    \item $(\forall A \in Form_X)(\{ A \} = \mathcal{E}(m, y_{(j,
    A)})= \mathcal{E}(m, w_A) = \mathcal{E}(m, u_A))$, for all $m \in Tree \setminus \{ \dag, \ddag \}$ and $j \in Ag$;
    \item $(\forall A \in Form_X)(A \in \mathcal{E}(m, z)
     \Leftrightarrow (\forall \xi \in m)(Proven(A) \in end(\xi)))$, for all $m \in Tree \setminus \{ \dag, \ddag \}$;
    \item $\mathcal{E}(m, t) = Form$, if $m \in Tree \setminus \{ \dag, \ddag \}$ and $t
    \in Pol \setminus(Pol_X \cup Y \cup W \cup U \cup \{ z \})$.
\end{itemize}
We start by mentioning a straightforward corollary to the above
definition:
\begin{lemma}\label{proven}
For all $m \in Tree$ and $t \in Pol$ it is true that $\{ A \in
Form \mid \vdash t\co A \} \subseteq \mathcal{E}(m,t)$.
\end{lemma}
\begin{proof}
This holds simply by the definition of $\mathcal{E}$ when $m \in
\{ \dag, \ddag \}$. If $m \in Tree \setminus \{ \dag, \ddag \}$,
then we have another obvious case for $t \in Pol \setminus(Pol_X
\cup Y \cup W \cup U \cup \{ z \})$.

If $t \in Pol_X$, and $\vdash t\co A$, then we have just a border
case in the definition of $\mathcal{E}(m,t)$, with $t\co A$
following from the empty conjunction of elements present in
$end(\xi)$ for every $\xi \in m$.

Finally, if $t \in Y \cup W \cup U \cup \{ z \}$, then $t \in
PVar$, therefore, by Lemma \ref{theorems}.3, we must have:
$$
\{ A \in Form \mid \vdash t\co A \} = \emptyset \subseteq
\mathcal{E}(m,t).
$$
\end{proof}

Note that since we know that for every $c \in PConst$ and every
instance $A$ of one of the axiom schemes in the list
\eqref{A0}--\eqref{A13}, it is true that $\vdash c\co A$  (by
\eqref{R3}), it follows, among other things, that the
above-defined function $\mathcal{E}$ satisfies the additional
normality condition on jstit models.

It is straightforward to check that $R$, as defined above, is a
preorder on $Tree$, using \eqref{A7} and \eqref{A8}. Let us
briefly look into why the future always matters constraint is
verified as well. Assume $m \in Tree$. If $m = \dag$, then it is
connected to all the elements in $Tree$ by both $\leq$ and $R$,
and if $m = \ddag$, then it is connected only to itself by both
$\leq$ and $R$, so these moments cannot falsify the constraint. So
let us assume that $m$ is a class of equivalence generated by some
element, say $m = [(\Gamma_1,\ldots, \Gamma_n)\alpha]_\equiv$. If
$m \leq m'$, then $m'$ must be an equivalence class as well, and
$(\Gamma_1,\ldots, \Gamma_n)\alpha$ must be an initial segment of
every element in $m'$, so that we may assume, wlog, that $m' =
[(\Gamma_1,\ldots, \Gamma_k)\alpha]_\equiv$ for some $k \geq n$.
In particular, if $n
> 1$, then $k - 1 > 0$. But then take an arbitrary $A \in Form_X$. If $KA \in \Gamma_n$, then, since
$(\Gamma_1,\ldots, \Gamma_k)\alpha$ is an element, $KA \in
\Gamma_k$. By Lemma \ref{theorems}.5 and maxiconsistency of
$\Gamma_k$ in $X$ we must have then $\Box KA \in \Gamma_k$. Now,
by definition of $\equiv$, we get $KA \in end(\tau)$ for any given
$\tau \in m'$. It follows then that $R(m, m')$ as desired.

We further check that the semantical constraints for $\mathcal{E}$
are verified:
\begin{lemma}\label{e}
The function $\mathcal{E}$, as defined above, satisfies both the
monotonicity of evidence and the evidence closure properties.
\end{lemma}
\begin{proof}
We start with the monotonicity of evidence. Assume $R(m,m')$ and
$t \in Pol$. If $m \in \{ \dag, \ddag \}$ then, by Lemma
\ref{proven}, $\mathcal{E}(m,t) = \{ A \in Form \mid \vdash t\co A
\} \subseteq \mathcal{E}(m',t)$ for any $m' \in Tree$.

Assume, further, that $m$ is an equivalence class of elements.
Then, since we have $R(m,m')$, $m'$ must be an equivalence class
of elements as well. Also, we are done if $m = m'$. On the other
hand, if $m \neq m'$, then consider $t$. If $t \in Pol \setminus
(Pol_X \cup \{ z \})$, then we must have $\mathcal{E}(m,t) =
\mathcal{E}(m',t)$ by definition, since $m, m' \in Tree \setminus
\{ \dag, \ddag \}$. If $t = z$, then take an arbitrary $A \in
\mathcal{E}(m,z)$. By the above definition of $\mathcal{E}$, this
means that $Proven(A) \in end(\xi)$ for every element $\xi$ of
$m$. By maxiconsistency of $end(\xi)$ in $X$ and \eqref{A11}, this
further means that $KProven(A) \in end(\xi)$ for every element
$\xi$ of $m$. Therefore, by $R(m,m')$, and the fact that $m, m'
\in Tree \setminus \{ \dag, \ddag \}$, we get that $KProven(A) \in
end(\tau)$ for every element $\tau$ of $m'$, whence, by
\eqref{A7}, it follows that $Proven(A) \in end(\tau)$ for every
element $\tau$ of $m'$. Therefore $A \in \mathcal{E}(m',z)$. Since
$A$ was arbitrary, this means that $\mathcal{E}(m,z) \subseteq
\mathcal{E}(m',z)$, as desired.

Finally, assume that $t \in Pol_X$ and take an arbitrary $A \in
\mathcal{E}(m,t)$. Then we can choose $t_1\co B_1,\ldots,t_n\co
B_n$ in such a way that for all $\xi \in m$ we have $t_1\co
B_1,\ldots,t_n\co B_n \in end(\xi)$, and, moreover, $\vdash
(t_1\co B_1\wedge\ldots \wedge t_n\co B_n) \to t\co A$. Since
$t_1\co B_1,\ldots,t_n\co B_n \in end(\xi)$, we know that $\{
t_1\co B_1,\ldots,t_n\co B_n \} \in Form_X$. We also know that,
for every $\xi \in m$, $end(\xi)$ is maxiconsistent in $X$.
Therefore, using Lemma \ref{elementaryconsistency}, we obtain,
successively:
\begin{align*}
&(\forall \xi \in m)(Kt_1\co B_1,\ldots,Kt_n\co B_n \in end(\xi))
&&\textup{(by
Lemma \ref{theorems}.1)}\\
&(\forall \xi \in m)((Kt_1\co B_1\wedge\ldots\wedge Kt_n\co B_n)
\in end(\xi)) &&\textup{(by
Lemma \ref{elementaryconsistency}.5)}\\
&(\forall \xi \in m)(K(t_1\co B_1\wedge\ldots\wedge t_n\co B_n)
\in end(\xi)) &&\textup{(by \eqref{A7})}
\end{align*}
From the latter it follows by $R(m,m')$ that $K(t_1\co
B_1\wedge\ldots\wedge t_n\co B_n) \in end(\tau)$ for all $\tau \in
m'$. We also know that for every $\tau \in m'$, $end(\tau)$ is
maxiconsistent in $X$ so that, applying Lemma
\ref{elementaryconsistency}, and \eqref{A7}, we get that $t_1\co
B_1,\ldots,t_n\co B_n \in end(\tau)$ for all $\tau \in m'$. Adding
this to our initial assumption that $\vdash (t_1\co
B_1\wedge\ldots \wedge t_n\co B_n) \to t\co A$, we obtain that $A
\in \mathcal{E}(m',t)$.

We turn now to the closure conditions. We verify the first two
conditions, and the third one can be verified in a similar way,
restricting attention to $t$ rather than considering both $s$ and
$t$. Let $s, t \in Pol$. We need to consider two cases:

\emph{Case 1}. $m \in \{ \dag, \ddag \}$. If $A \in
\mathcal{E}(m,s)$, then $\vdash s\co A$. Therefore, by \eqref{A6},
we must also have $\vdash (s + t)\co A$ so that $A \in
\mathcal{E}(m,s + t)$. Similarly, if $A \in \mathcal{E}(m,t)$,
then also $A \in \mathcal{E}(m,s + t)$ and the closure constraint
(b) is verified. If, on the other hand, it is true that for some
$A, B \in Form$ we have both $A \to B \in \mathcal{E}(m,s)$ and $A
\in \mathcal{E}(m,t)$, then, again, this means that both $\vdash
s\co A \to B$ and $\vdash t\co A$. By \eqref{A4}, it follows that
$\vdash s\times t\co B$ and, therefore, also $B \in
\mathcal{E}(m,s\times t)$, so that the closure condition (a) is
also verified.

\emph{Case 2}. $m \in Tree \setminus \{ \dag, \ddag \}$. If $s +
t, s\times t \notin Pol_X$, then we have:
$$
\mathcal{E}(m, s + t) = \mathcal{E}(m, s \times t) = Form,
$$
so that all the closure conditions are verified trivially.
Therefore, assume that

\noindent$s + t, s\times t \in Pol_X$. If $A \in Form$ and  $A \in
\mathcal{E}(m,s)$, then  we can choose $t_1\co B_1,\ldots,t_n\co
B_n$ such that for all $\xi \in m$ we have both $t_1\co
B_1,\ldots,t_n\co B_n \in end(\xi)$ and
$$
\vdash (t_1\co B_1\wedge\ldots \wedge t_n\co B_n) \to s\co A.
$$
By \eqref{A0} and \eqref{A6} we get then that $\vdash (t_1\co
B_1\wedge\ldots \wedge t_n\co B_n) \to s + t\co A$, which means
that $A \in \mathcal{E}(m,s + t)$. Similarly, if $A \in
\mathcal{E}(m,t)$, then $A \in \mathcal{E}(m,s + t)$ as well, and
closure condition (b) is verified.

On the other hand, if $A, B \in Form$ and we have both both $A \to
B \in \mathcal{E}(m,s)$ and $A \in \mathcal{E}(m,t)$, then we can
choose $t_1\co B_1,\ldots,t_n\co B_n$ and also $s_1\co
C_1,\ldots,s_k\co C_k$ such that for every $\xi \in m$ we have all
of the following:
\begin{align*}
    &t_1\co B_1,\ldots,t_n\co B_n, s_1\co
C_1,\ldots,s_k\co C_k \in end(\xi);\\
&\vdash (t_1\co B_1\wedge\ldots \wedge t_n\co B_n) \to t\co A;\\
&\vdash (s_1\co C_1\wedge\ldots \wedge s_k\co C_k) \to s\co (A \to
B);
\end{align*}
It follows then by \eqref{A0} and \eqref{A4} that:
\begin{align*}
  &(t_1\co B_1\wedge\ldots \wedge t_n\co B_n \wedge s_1\co C_1\wedge\ldots \wedge s_k\co
  C_k)\to s\times t\co B,
\end{align*}
so that $B \in \mathcal{E}(m,s \times t)$ follows and closure
condition (a) is verified.
\end{proof}

\subsection{$Act$ and $V$}

It only remains to define $Act$ and $V$ for our canonical model,
and we define them as follows:
\begin{itemize}
    \item $(m,h) \in V(p) \Leftrightarrow p \in end(m \cap h)$,
    for all $p \in Var$;
    \item $Act(\dag, (\dag,\ddag)) = Act(\ddag, (\dag,\ddag)) = \emptyset$;
    \item $Act(\dag, h) = \{ z \}$, if $h \neq (\dag,\ddag)$;
    \item $Act(m,h) = \{ z \} \cup \{ y_{(j, A)} \mid Prove(j, A) \wedge \neg\Box Prove(j, A)
    \in \Gamma_1 \} \cup$

     $\qquad\qquad\qquad\qquad\qquad\qquad\cup\{ u_A \mid \Box Prove(j, A) \in \Gamma_1
    \}$, if $m \cap h = (\Gamma_1,\ldots, \Gamma_n)\uparrow$;
     \item $Act(m,h) = \{ z \} \cup \{ y_{(j, A)} \mid Prove(j, A) \wedge \neg\Box Prove(j, A)
    \in \Gamma_1 \} \cup$

    $\qquad\qquad\qquad\qquad\qquad\qquad\cup \{ w_A \mid \Box Prove(j, A) \in \Gamma_1
    \}$, if $m \cap h = (\Gamma_1,\ldots, \Gamma_n)\downarrow$.
\end{itemize}

We begin by establishing some consequences of the above
definition:
\begin{lemma}\label{intersections}
The following statements are true:
\begin{enumerate}
\item If $(\Gamma)\alpha$ is an element, then:
$$
\bigcap_{h \in
H_{[(\Gamma)\alpha]_\equiv}}(Act([(\Gamma)\alpha]_\equiv, h) = \{
z \}.
$$

\item If $n > 1$ and $(\Gamma_1,\ldots, \Gamma_n)\alpha$ is an
element and $g \in H_{[(\Gamma_1,\ldots, \Gamma_n)\alpha]_\equiv}$
is arbitrary, then:
$$
\bigcap_{h \in H_{(\Gamma_1,\ldots,
\Gamma_n)\alpha]_\equiv}}(Act([(\Gamma_1,\ldots,
\Gamma_n)\alpha]_\equiv, h) = Act([(\Gamma_1,\ldots,
\Gamma_n)\alpha]_\equiv, g).
$$
\end{enumerate}
\end{lemma}
\begin{proof}
(Part 1). Set $m := [(\Gamma)\alpha]_\equiv$. It is clear from the
definition of $Act$ that

\noindent$z \in \bigcap_{h \in H_m}(Act(m,h))$, so that we only
need to show that $z$ is the only member in this intersection. The
other elements of $Act$, according to the definition, can have one
of the following forms: either $y_{(j, A)}$, or $u_A$, or $w_A$,
for some $A \in Form_X$ and $j \in Ag$. We know, further, that
both $(\Gamma)\alpha \equiv (\Gamma)\uparrow$ and $(\Gamma)\alpha
\equiv (\Gamma)\downarrow$.\footnote{One of these two elements
even coincides with $(\Gamma)\alpha$, but we cannot tell, which
one.} Then, using Lemma \ref{hist}.2, take any $h, h' \in H_m$ for
which $h \cap m = (\Gamma)\uparrow$ and $h' \cap m =
(\Gamma)\downarrow$. By definition, $Act(m, h)$ is disjoint from
$\{ w_A \mid A \in Form_X \}$ whereas $Act(m, h')$ is disjoint
from $\{ u_A \mid A \in Form_X \}$, therefore, $\bigcap_{h \in
H_m}(Act(m,h))$ must be disjoint from $\{ w_A \mid A \in Form_X \}
\cup \{ u_A \mid A \in Form_X \}$. Finally, consider a variable of
the form $y_{(j, A)}$ for arbitrary $A \in Form_X$ and $j \in Ag$.
If $y_{(j, A)} \in \bigcap_{h \in H_m}(Act(m,h))$, then recall
that for every $(\Delta)\alpha \in m$ there exists, by Lemma
\ref{hist}.2, a history $h_\Delta \in H_m$ such that
$(\Delta)\alpha = m \cap h_\Delta$. This means that $Prove(j, A)
\wedge \neg\Box Prove(j, A) \in \Delta$ for every $(\Delta)\alpha
\in m$, and thus, by maxiconsistency of $\Delta$ in $X$ and Lemma
\ref{elementaryconsistency}.5, that $Prove(j, A), \neg\Box
Prove(j, A) \in \Delta$ for every $(\Delta)\alpha \in m$. In
particular, we will have $Prove(j, A), \neg\Box Prove(j, A) \in
\Gamma$. Consider then the following set of formulas in $X$:
$$
\Xi = \{ B \mid \Box B \in \Gamma \} \cup \{ \neg Prove(j,A) \}.
$$
$\Xi$ is consistent, for otherwise we would have:
$$
\vdash (B_1 \wedge\ldots \wedge B_k) \to Prove(j,A),
$$
for some $B_1,\ldots,B_k$ such that $\Box B_1,\ldots,\Box B_k$ are
all in $\Gamma$. Since $\Box$ is an S5-modality, we would obtain
that
$$
\vdash (\Box B_1 \wedge\ldots \wedge \Box B_k) \to \Box
Prove(j,A),
$$
whence, by maixiconsistency of $\Gamma$ in $X$, it would follow
that $\Box Prove(j,A) \in \Gamma$, and the latter, given that also
$\neg\Box Prove(j, A) \in \Gamma$, would contradict $\Gamma$'s
maxiconsistency. Therefore, $\Xi$ is consistent and we can extend
$\Xi$ to a set $\Theta \subseteq Form_X$, which is maxiconsistent
in $X$. By definition, we will have $(\Gamma)\alpha \equiv
(\Theta)\alpha$, and thus $(\Theta)\alpha \in m$. But we will also
have $\neg Prove(j,A) \in \Theta$ which contradicts our assumption
that $Prove(j, A) \in \Delta$ for every $(\Delta)\alpha \in m$.
This contradiction shows that no proof variable of the form
$y_{(j, A)}$ is in $\bigcap_{h \in H_m}(Act(m,h))$. Therefore,
finally, we get our claim that $\bigcap_{h \in H_m}(Act(m,h)) = \{
z \}$ verified.

(Part 2). We set $m := [(\Gamma_1,\ldots,
\Gamma_n)\alpha]_\equiv$. It will suffice to show that, for all
$h,h' \in H_m$, we have $Act(m,h) = Act(m,h')$. We know that for
some appropriate $\Delta_1,\ldots, \Delta_n, \Theta_1,\ldots,
\Theta_n$ we will have:
$$
m \cap h = (\Delta_1,\ldots, \Delta_n)\alpha,
$$
and:
$$
m \cap h' = (\Theta_1,\ldots, \Theta_n)\alpha.
$$
Since the length of $m$ is greater than $1$, we know that all
elements in $m$ share the same greatest proper initial segment, so
that we have:
$$
\Gamma_i = \Delta_i = \Theta_i
$$
for all $i < n$, and, in particular:
$$
\Gamma_1 = \Delta_1 = \Theta_1.
$$
Now it is clear from the definition of $Act$, that $Act(m,h)$ and
$Act(m,h')$ are completely determined by $\alpha$, $\Delta_1$ and
$\Theta_1$, respectively, therefore, it follows that $Act(m,h) =
Act(m,h')$.
\end{proof}

 We now check that the remaining semantic constraints on normal jstit models:

\begin{lemma}\label{act}
The canonical model, as defined above, satisfies the constraints
as to the expransion of presented proofs, no new proofs
guaranteed, presenting a new proof makes histories divide, and
epistemic transparency of presented proofs.
\end{lemma}
\begin{proof}
We consider the expansion of presented proofs first. Let $m' < m$
and let $h \in H_m$. Then $m' \neq \ddag$, since $\ddag$ has no
$<$-successors. If $m' = \dag$ and $m = \ddag$, then $h$ must be
$(\dag,\ddag)$ and we have $Act(\dag, (\dag,\ddag)) = Act(\ddag,
(\dag,\ddag)) = \emptyset$, so that the expansion of presented
proofs holds. If $m' = \dag$ and $m$ is an equivalence class of
elements, then $h \neq (\dag,\ddag)$, and we have $Act(\dag, h) =
\{ z \}$ and $z \in Act(m,h)$. Finally, if $m'$ is an equivalence
class of elements, then $m$ is also an equivalence class of
elements. In this case, $m \cap h$ must be of the form
$(\Gamma_1,\ldots, \Gamma_n)\alpha$ for the respective
$\Gamma_1,\ldots, \Gamma_n \subseteq Form_X$ and $\alpha \in \{
\uparrow, \downarrow \}$. But then, for some $k \leq n$, $m' \cap
h$ must be of the form $(\Gamma_1,\ldots, \Gamma_k)\alpha$. Since
the extension of both $Act(m,h)$ and $Act(m',h)$ is determined by
$\alpha$ and $\Gamma_1$, and these are shared by both $m \cap h$
and $m' \cap h$, it follows that $Act(m,h) = Act(m',h)$ and thus
$Act(m,h) \subseteq Act(m',h)$.

We consider next the no new proofs guaranteed constraint. Let $m
\in Tree$. If $m \in \{ \dag,\ddag \}$, then $\bigcap_{h \in
H_m}(Act(m,h)) = \bigcup_{m' < m, h \in H_m}(Act(m',h)) =
\emptyset$ and the constraint is trivially satisfied. If $m \in
Tree \setminus \{ \dag, \ddag \}$, then we need to distinguish
between two cases:

\emph{Case 1}. The length of $m$ equals $1$. Then $m$ is of the
form $[(\Gamma)\alpha]_\equiv$ for the respective $\Gamma
\subseteq Form_X$ and $\alpha \in \{ \uparrow, \downarrow \}$. By
Lemma \ref{intersections}.1, we have then that $\bigcap_{h \in
H_m}(Act(m,h)) = \{ z \}$. On the other hand, note that the only
$<$-predecessor of $[(\Gamma)\alpha]_\equiv = m$ must be $\dag$
and therefore, by definition of $Act$, we get that $\bigcup_{h \in
H_m}(Act(\dag,h)) = \{ z \}$ so that the no new proofs guaranteed
constraint is verified for $m$.

\emph{Case 2}. The length of $m$ is greater than $1$. Then $m$
must be of the form $[(\Gamma_1,\ldots, \Gamma_n, \Gamma_{n +
1})\alpha]_\equiv$ for the respective $\Gamma_1,\ldots, \Gamma_n,
\Gamma_{n + 1} \subseteq Form_X$, $n > 0$, and $\alpha \in \{
\uparrow, \downarrow \}$. Then we choose, by Lemma \ref{hist}.2,
an arbitrary $g \in H_m$ such that $m \cap g = (\Gamma_1,\ldots,
\Gamma_n, \Gamma_{n + 1})\alpha$. For this $g$ we get, using Lemma
\ref{intersections}.2:
\begin{align*}
    &\bigcap_{h \in
H_m}(Act([(\Gamma_1,\ldots, \Gamma_n, \Gamma_{n +
1})\alpha]_\equiv,h)) = Act([(\Gamma_1,\ldots, \Gamma_n, \Gamma_{n
+
1})\alpha]_\equiv,g)\\
&=\{ z \} \cup \{ y_{(j, A)} \mid Prove(j, A) \wedge \neg\Box
Prove(j, A) \in \Gamma_1 \} \cup \{ w_A \mid \Box Prove(j, A) \in
\Gamma_1 \}\\
&\qquad\qquad\qquad\qquad = Act([(\Gamma_1)\alpha]_\equiv,g)\\
&\qquad\qquad\qquad\qquad\subseteq \bigcup_{m' < m, h \in
H_m}(Act(m',h)),
\end{align*}
since $[(\Gamma_1)\alpha]_\equiv < [(\Gamma_1,\ldots, \Gamma_n,
\Gamma_{n + 1})\alpha]_\equiv$.

We turn next to the presenting a new proof makes histories divide
constraint. Consider an $m, m' \in Tree$ such that $m < m'$ and
arbitrary $h, h' \in H_{m'}$. If $m = \ddag$, then the constraint
is verified trivially since $\ddag$ has no $<$-successors. If $m =
\dag$ and $m' = \ddag$, then we must have $h = h' = (\dag, \ddag)$
and the constraint is verified trivially. If $m = \dag$ and $m'
\neq \ddag$, then both $h$ and $h'$ are different from $(\dag,
\ddag)$, which means that $Act(\dag, h) = Act(\dag, h') = \{ z
\}$, and the constraint is again verified. Finally, if $m \in Tree
\setminus \{ \dag, \ddag \}$, then we must have $m =
[(\Gamma_1,\ldots, \Gamma_n)\alpha]_\equiv$ for some appropriate
$\Gamma_1,\ldots, \Gamma_n,\alpha$. But then, since $m' > m$, it
must be that $m' = [(\Gamma_1,\ldots, \Gamma_k)\alpha]_\equiv$ for
some $k > n$ (so that, among other things, we know that $k > 1$).
Now, given that $h, h' \in H_{m'}$, this means that for
appropriate $\Delta, \Delta' \subseteq Form_X$ we must have $h
\cap m' = (\Gamma_1,\ldots, \Gamma_{k - 1}, \Delta)\alpha$ and $h'
\cap m' = (\Gamma_1,\ldots, \Gamma_{k - 1}, \Delta')\alpha$,
which, in turn, means that:
$$
h \cap m = h' \cap m = (\Gamma_1,\ldots, \Gamma_n)\alpha.
$$
It follows, by definition of $Act$, that in this case $Act(m,h) =
Act(m,h')$, and the constraint is verified.

It remains to check the epistemic transparency of presented proofs
constraint. Assume that $m, m' \in Tree$ are such that $R(m,m')$.
If we have $m \in \{ \dag,\ddag \}$, then, by definition, we must
have $\bigcap_{h \in H_m}(Act(m,h)) = \emptyset$, and the
constraint is verified in a trivial way. If, on the other hand, $m
\in Tree \setminus \{ \dag, \ddag \}$, then, by $R(m,m')$, we must
also have $m' \in Tree \setminus \{ \dag, \ddag \}$. We have then
two cases to consider:

\emph{Case 1}. The length of $m$ equals $1$. Then, by Lemma
\ref{intersections}.1, we know that

\noindent$\bigcap_{h \in H_m}(Act(m,h)) = \{ z \}$. It is also
obvious, by the fact that $m' \in Tree \setminus \{ \dag, \ddag
\}$, that $z \in \bigcap_{h \in H_m'}(Act(m',h))$ and thus the
constraint is satisfied.

\emph{Case 2}. The length of $m$ is greater than $1$. Then $m =
[(\Gamma_1,\ldots, \Gamma_n)\alpha]_\equiv$ for appropriate
$\Gamma_1,\ldots, \Gamma_n,\alpha$, and, since we have $R(m,m')$,
we must also have $m' = [(\Gamma_1,\Delta_1\ldots,
\Delta_k)\alpha]_\equiv$ for appropriate $\Delta_1\ldots,
\Delta_k$. We assume that in fact $\alpha = \downarrow$, the other
subcase is similar. Using Lemma \ref{hist}.2, we choose $g \in
H_m$ and $g' \in H_{m'}$ in such a way that $m \cap g =
(\Gamma_1,\ldots, \Gamma_n)\downarrow$ and $m' \cap g' =
(\Gamma_1,\Delta_1\ldots, \Delta_k)\downarrow$. We get then, by
Lemma \ref{intersections}.2:
\begin{align*}
    &\bigcap_{h \in H_m}(Act(m,h)) = Act(m,g)\\
    &=\{ z \} \cup \{ y_{(j, A)} \mid Prove(j, A) \wedge \neg\Box
Prove(j, A) \in \Gamma_1 \} \cup \{ w_A \mid \Box Prove(j, A) \in
\Gamma_1 \}\\
&\qquad\qquad\qquad\qquad\qquad = Act(m',g')\\
&\qquad\qquad\qquad\qquad\qquad = \bigcap_{h' \in
H_{m'}}(Act(m',h')).
\end{align*}
\end{proof}

\subsection{The truth lemma}
It follows from Lemmas \ref{leq}--\ref{act} that our above-defined
canonical model is in fact a normal jstit model. Now we need to
supply a truth lemma:
\begin{lemma}\label{truth}
Let $A \in Form_X$, let $m \in Tree \setminus \{ \dag,\ddag \}$,
and let $h \in H_m$. Then:
$$
\mathcal{M},m,h \models A \Leftrightarrow A \in end(m \cap h).
$$
\end{lemma}
\begin{proof}
As is usual, we prove the lemma by induction on the construction
of $A$. The basis of induction with $A = p \in Var$ we have by
definition of $V$, whereas Boolean cases for the induction step
are trivial. We treat the modal cases:

\emph{Case 1}. $A = \Box B$. If $\Box B \in end(m \cap h)$, then
note that for every $h' \in H_m$ we must have $m \cap h' \in m$ so
that $m \cap h' \equiv m \cap h$. By definition of $\equiv$ and
the fact that $m \in Tree \setminus \{ \dag,\ddag \}$, we must
have then $B \in (m \cap h')$ for all $h' \in H_m$ and thus, by
induction hypothesis, we obtain that $\mathcal{M},m,h \models \Box
B$. If, on the other hand, $\Box B \notin end(m \cap h)$, we need
to consider then two subcases:

\emph{Case 1.1}. The length of $m$ equals $1$. We must have then
$m \cap h = (\Gamma)\alpha$ for some appropriate $\Gamma$ and
$\alpha$ so that $\Gamma = end(m \cap h)$. Then the set
$$
\Xi = \{ \Box C \mid \Box C \in \Gamma \} \cup \{ \neg B \}
$$
must be consistent, since otherwise we would have
$$
\vdash (\Box C_1\wedge\ldots\wedge\Box C_n) \to B
$$
for some $\Box C_1,\ldots,\Box C_n \in \Gamma$, whence, since
$\Box$ is an S5-modality, we would get
$$
\vdash (\Box C_1\wedge\ldots\wedge\Box C_n) \to \Box B,
$$
which would mean that $\Box B \in \Gamma$, contrary to our
assumption. Therefore, $\Xi$ is consistent and we can extend $\Xi$
to a set $\Delta \in Form_X$ which is maxiconsistent in $X$. Of
course, in this case $B \notin \Delta$. We will have then that
$(\Delta)\alpha$ is an element, and, by definition of $\equiv$,
that $(\Gamma)\alpha \equiv (\Delta)\alpha$. By Lemma
\ref{hist}.2, for some $h' \in H_m$ we will have $(\Delta)\alpha =
m \cap h'$ and, therefore, $\Delta = end(m \cap h')$. Since $B
\notin \Delta$, it follows, by induction hypothesis, that
$\mathcal{M},m,h' \not\models B$, hence $\mathcal{M},m,h
\not\models \Box B$ as desired.

\emph{Case 1.2}. The length of $m$ is greater than $1$. We must
have then $m \cap h = (\Gamma_1,\ldots,\Gamma_n,\Gamma)\alpha$ for
some appropriate $n > 0$, $\Gamma_1,\ldots,\Gamma_n,\Gamma$ and
$\alpha$ so that $\Gamma = end(m \cap h)$. We then define $\Delta$
as in Case 1.1 so that we have
$(\Gamma_1,\ldots,\Gamma_n,\Gamma)\alpha \equiv
(\Gamma_1,\ldots,\Gamma_n,\Delta)\alpha$ and show that for any $h'
\in H_m$ such that $(\Gamma_1,\ldots,\Gamma_n,\Delta)\alpha = m
\cap h'$ and, \emph{eo ipso}, $\Delta = end(m \cap h')$, we will
have $\mathcal{M},m,h' \not\models B$, whence $\mathcal{M},m,h
\not\models \Box B$ as desired. The only new ingredient is that
now we need to supply a proof that
$(\Gamma_1,\ldots,\Gamma_n,\Delta)\alpha$ is actually an element.
Well, if for any $C \in Form_X$ we have that $KC \in \Gamma_n$,
then, since $(\Gamma_1,\ldots,\Gamma_n,\Gamma)\alpha$ is an
element, we will have $KC \in \Gamma$, whence, by maxiconsistency
of $\Gamma$ in $X$ and Lemma \ref{theorems}.5, $\Box KC \in
\Gamma$, and since every boxed formula from $\Gamma$ is also in
$\Delta$, we get that $\Box KC \in \Delta$, whence $KC \in \Delta$
by maxiconsistency of $\Delta$ in $X$ and S5 reasoning for $\Box$.
Further, if we have $Prove(j, C) \in \Gamma_1$ for $j \in Ag$,
then, since $(\Gamma_1,\ldots,\Gamma_n,\Gamma)\alpha$ is an
element, we will have $Proven(C) \in \Gamma$, whence, by
maxiconsistency of $\Gamma$ in $X$ and Lemma \ref{theorems}.2,
$\Box Proven(C) \in \Gamma$, and since every boxed formula from
$\Gamma$ is also in $\Delta$, we get that $\Box Proven(C) \in
\Delta$, whence $\Box Proven(C) \in \Delta$ by maxiconsistency of
$\Delta$ in $X$ and S5 reasoning for $\Box$. Finally, if $C \in
Form_X$ and $j \in Ag$, then $K\Box\neg Prove(j, C) \in \Gamma$,
because $(\Gamma_1,\ldots,\Gamma_n,\Gamma)\alpha$ is an element,
whence $\Box K\Box\neg Prove(j, C) \in \Gamma$ by \eqref{A8} and
maxiconsistency of $\Gamma$ in $X$. And since every boxed formula
from $\Gamma$ is also in $\Delta$, we get that $\Box K\Box\neg
Prove(j, C) \in \Delta$ as well, hence $K\Box\neg Prove(j, C) \in
\Delta$ by \eqref{A1} and maxiconsistency of $\Delta$ in $X$.

\emph{Case 2}. $A = [j]B$ for some $j \in Ag$. Then, if $[j]B \in
end(m \cap h)$, by definition of $Choice$ and the fact that $m \in
Tree \setminus \{ \dag,\ddag \}$ we must have:
$$
Choice^m_j(h) = \{ h' \mid h' \in H_m,\,
    (\forall C \in Form_X)([j]C \in end(h \cap m) \Rightarrow C \in end(h' \cap
    m))\}.
$$
Therefore, if $h' \in Choice^m_j(h)$, then we must have $B \in
end(h' \cap m)$, and further, by induction hypothesis, that
$\mathcal{M},m,h' \models B$, so that we get $\mathcal{M},m,h
\models [j]B$. On the other hand, if $[j]B \notin end(m \cap h)$,
we need to consider then two subcases:

\emph{Case 2.1}. The length of $m$ equals $1$. We must have then
$m \cap h = (\Gamma)\alpha$ for some appropriate $\Gamma$ and
$\alpha$ so that $\Gamma = end(m \cap h)$. Then the set
$$
\Xi = \{ [j]C \mid [j]C \in \Gamma \} \cup \{ \neg B \}
$$
must be consistent, since otherwise we would have
$$
\vdash ([j]C_1\wedge\ldots\wedge[j]C_n) \to B
$$
for some $[j]C_1,\ldots,[j]C_n \in \Gamma$, whence, since $[j]$ is
an S5-modality, we would get
$$
\vdash ([j]C_1\wedge\ldots\wedge[j]C_n) \to [j]B,
$$
which would mean that $[j]B \in \Gamma$, contrary to our
assumption. Therefore, $\Xi$ is consistent and we can extend $\Xi$
to a set $\Delta \subseteq Form_X$ which is maxiconsistent in $X$.
Of course, in this case $B \notin \Delta$. We will have then that
$(\Delta)\alpha$ is an element.

Now, if $D \in Form_X$ is such that $\Box D \in \Gamma$, then, by
\eqref{A2} and maxiconsistency of $\Gamma$ in $X$, we know that
$[j]D \in \Gamma$, so that also $[j]D \in \Delta$, and hence, by
\eqref{A1} and maxiconsistency of $\Delta$ in $X$, $D \in \Delta$.
We have thus shown that:
 \begin{equation}\label{E:t1}
(\forall D \in Form_X)(\Box D \in \Gamma \Rightarrow D \in
\Delta),
\end{equation}
 and it follows that $(\Gamma)\alpha
\equiv (\Delta)\alpha$ by definition of $\equiv$. By Lemma
\ref{hist}.2, for some $h' \in H_m$ we will have $(\Delta)\alpha =
m \cap h'$ and, therefore, $\Delta = end(m \cap h')$. Also, since
$\Delta$ contains all the $[j]$-modalized formulas from $\Gamma$,
we know that for any such $h'$ we will have $h' \in
Choice^m_j(h)$. Since $B \notin \Delta$, it follows, by induction
hypothesis, that $\mathcal{M},m,h' \not\models B$, hence
$\mathcal{M},m,h \not\models [j]B$ as desired.

\emph{Case 2.2}. The length of $m$ is greater than $1$. We must
have then $m \cap h = (\Gamma_1,\ldots,\Gamma_n,\Gamma)\alpha$ for
some appropriate $n > 0$, $\Gamma_1,\ldots,\Gamma_n,\Gamma$ and
$\alpha$ so that $\Gamma = end(m \cap h)$. We then define $\Delta$
as in Case 2.1 so that we have
$(\Gamma_1,\ldots,\Gamma_n,\Gamma)\alpha \equiv
(\Gamma_1,\ldots,\Gamma_n,\Delta)\alpha$ and show that for any $h'
\in H_m$ such that $(\Gamma_1,\ldots,\Gamma_n,\Delta)\alpha = m
\cap h'$ and, \emph{eo ipso}, $\Delta = end(m \cap h')$, we will
have both $h' \in Choice^m_j(h)$ and $\mathcal{M},m,h' \not\models
B$, whence $\mathcal{M},m,h \not\models [j]B$ as desired. Again, a
separate argument for $(\Gamma_1,\ldots,\Gamma_n,\Delta)\alpha$
being an element needs to be supplied, and it can be done in the
same way as in Case 1.2, given that by \eqref{E:t1} and S5
properties of $\Box$ we know that every boxed formula from
$\Gamma$ is also in $\Delta$.

\emph{Case 3}. $A = KB$. Assume that $KB \in end(m \cap h)$. We
clearly have then $m = [(m \cap h)]_\equiv$. Hence, by definition
of $R$ and the fact that $m \in Tree \setminus \{ \dag,\ddag \}$
we must have for every $m' \in Tree$:
$$
R(m,m') \Rightarrow (\forall \tau \in m')(\forall C \in Form_X)(KC
\in end(m \cap h) \Rightarrow KC \in end(\tau)).
$$
Therefore, if $R(m,m')$ and $h' \in H_{m'}$ is arbitrary, then, of
course, $(h' \cap m') \in m'$ so that $KB \in end(h' \cap m')$,
and, further, $B \in end(h' \cap m')$ by S4 reasoning for $K$.
Therefore, by induction hypothesis, we get that $\mathcal{M},m',h'
\models B$, whence $\mathcal{M},m,h \models KB$. On the other
hand, if $KB \notin end(m \cap h)$, we need to consider then two
subcases:

\emph{Case 3.1}. The length of $m$ equals $1$. We must have then
$m \cap h = (\Gamma)\alpha$ for some appropriate $\Gamma$ and
$\alpha$ so that $\Gamma = end(m \cap h)$. Then the set
$$
\Xi = \{ KC \mid KC \in \Gamma \} \cup \{ \neg\Box B \}
$$
must be consistent, since otherwise we would have
$$
\vdash (KC_1\wedge\ldots\wedge KC_n) \to \Box B
$$
for some $KC_1,\ldots,KC_n \in \Gamma$, whence, since $K$ is an
S4-modality, we would get
$$
\vdash (KC_1\wedge\ldots\wedge KC_n) \to K\Box B,
$$
which would mean that $K\Box B \in \Gamma$, hence, by \eqref{A1},
\eqref{A7} and maxiconsistency of $\Gamma$ in $X$, that $KB \in
\Gamma$, contrary to our assumption. Therefore, $\Xi$ is
consistent and we can extend $\Xi$ to a set $\Delta \subseteq
Form_X$ which is maxiconsistent in $X$. Of course, in this case
$\Box B \notin \Delta$. We will have then that $(\Delta)\alpha$ is
an element. So we set $m' = [(\Delta)\alpha]_\equiv$. Assume that
$(\Delta')\alpha' \equiv (\Delta)\alpha$. Then every boxed formula
from $\Delta$ will be in $\Delta'$. In particular, whenever $KC
\in \Delta$, then also $\Box KC \in \Delta$ and thus $KC \in
\Delta'$, by Lemma \ref{theorems}.5 and maxiconsistency of
$\Delta$ in $X$. Therefore, whenever $KC \in \Gamma$ and $\tau \in
m' = [(\Delta)\alpha]_\equiv$, we have that $KC \in end(\tau)$ so
that we must have $R(m,m')$. On the other hand, since $\Box B
\notin \Delta$, then, by Case 1, there must be a $\tau \in m'$
such that $B \notin end(\tau)$. But then, by Lemma \ref{hist}.2,
we can choose $h' \in H_{m'}$ in such a way that $\tau = m' \cap
h'$, and we get that $B \notin end(m' \cap h')$. Therefore, by
induction hypothesis, we get $\mathcal{M},m',h' \not\models B$. In
view of the fact that also $R(m,m')$, this means that
$\mathcal{M},m,h \not\models KB$ as desired.

\emph{Case 3.2}. The length of $m$ is greater than $1$. We must
have then $m \cap h = (\Gamma_1,\ldots,\Gamma_n,\Gamma)\alpha$ for
some appropriate $n > 0$, $\Gamma_1,\ldots,\Gamma_n,\Gamma$ and
$\alpha$ so that $\Gamma = end(m \cap h)$. We then define $\Delta$
as in Case 3.1 and consider $m' =
[(\Gamma_1,\Delta)\alpha]_\equiv$. We get then $R(m,m')$
immediately by definition of $R$. Just as in Case 3.1, we will use
the fact that $\Box B \notin \Delta$ to find $\tau \in m'$ and $h'
\in H_{m'}$ so that $\tau = m' \cap h'$ and $B \notin end(\tau)$.
It will follow by induction hypothesis that $\mathcal{M},m',h'
\not\models B$, hence, given that $R(m,m')$, that $\mathcal{M},m,h
\not\models KB$.

The only new ingredient is that now we need to supply a proof that
$(\Gamma_1,\Delta)\alpha$ is actually an element. Well, if for any
$C \in Form_X$ we have that $KC \in \Gamma_1$, then, since
$(\Gamma_1,\ldots,\Gamma_n,\Gamma)\alpha$ is an element, we will
have $KC \in \Gamma$, whence $KC \in \Delta$. Further, if we have
$Prove(j, C) \in \Gamma_1$ for $j \in Ag$, then, since
$(\Gamma_1,\ldots,\Gamma_n,\Gamma)\alpha$ is an element, we will
have $Proven(C) \in \Gamma$, whence, by maxiconsistency of
$\Gamma$ in $X$ and \eqref{A11}, $KProven(C) \in \Gamma$, and
since every $K$-modalized formula from $\Gamma$ is also in
$\Delta$, we get that $KProven(C) \in \Delta$, whence $Proven(C)
\in \Delta$ by maxiconsistency of $\Delta$ in $X$ and S4 reasoning
for $K$. Finally, if $C \in Form_X$ and $j \in Ag$, then
$K\Box\neg Prove(j, C) \in \Gamma$, because
$(\Gamma_1,\ldots,\Gamma_n,\Gamma)\alpha$ is an element. And since
every $K$-modalized formula from $\Gamma$ is also in $\Delta$, we
get that $K\Box\neg Prove(j, C) \in \Delta$ as well.

\emph{Case 4}. $A = t\co B$ for some $t \in Pol_X$. Note that by
\eqref{A0} we know that $\vdash t\co B \to t\co B$. Therefore, if
$t\co B \in end(m \cap h)$, we will have $A \in \mathcal{E}(m,t)$
by definition. Also, by \eqref{A5} and maxiconsistency of $end(m
\cap h)$ in $X$, we will have $KB \in end(m \cap h)$. Therefore,
by Case 3, we will have that $\mathcal{M},m,h \models KB$ and
further, by $A \in \mathcal{E}(m,t)$, that $\mathcal{M},m,h
\models t\co B$. On the other hand, if $t\co B \notin end(m \cap
h)$, then for no

\noindent$t_1\co B_1,\ldots, t_n\co B_n \in end(m \cap h)$ can it
be that:
$$
\vdash (t_1\co B_1 \wedge\ldots\wedge t_n\co B_n) \to t\co B,
$$
for in this case we would also have $t\co B \in end(m \cap h)$ by
maxiconsistency of $end(m \cap h)$ in $X$. Therefore, we must have
$A \notin \mathcal{E}(m,t)$ so that $\mathcal{M},m,h \not\models
t\co B$.

\emph{Case 5}. $A = Proven(B)$. Assume that $Proven(B) \in end(m
\cap h)$. Then, by Lemma \ref{theorems}.1 and maxiconsistency of
$end(m \cap h)$, we will also have $\Box Proven(B) \in end(m \cap
h)$. Now, choose an arbitrary $\xi \in m$. We know that $end(\xi)
\equiv end(m \cap h)$, therefore, we must have $Proven(B) \in
end(\xi)$ by definition of $\equiv$, which means that $B \in
\mathcal{E}(m,z)$. We also have $z \in Act(m,h')$ for all $h' \in
H_m$ and we will have $KB \in end(m \cap h)$ by \eqref{A11} so
that we have $\mathcal{M},m,h \models z\co B$ by Case 3 and
induction hypothesis.\footnote{Note that sentences like $z\co B$
are not covered by our induction since $z \notin Pol_X$; but
sentences like $KB$ are covered since $B \in Form_X$. This is the
reason why our argument invokes Case 3 rather than Case 4.} It
follows then that $\mathcal{M},m,h \models Proven(B)$. On the
other hand, assume that $Proven(B) \notin end(m \cap h)$. Then we
have to consider two subcases:

\emph{Case 5.1}. The length of $m$ equals 1. Then, since
$Proven(B) \notin end(m \cap h)$, we will have $B \notin
\mathcal{E}(m,z)$ by definition. Also, by Lemma
\ref{intersections}.1, we know that

\noindent$\bigcap_{h' \in H_m}Act(m,h') = \{ z \}$. It follows
that $\mathcal{M},m,h \not\models Proven(B)$.

\emph{Case 5.2}. The length of $m$ is greater than 1.  We must
have then $m \cap h = (\Gamma_1,\ldots,\Gamma_n,\Gamma)\alpha$ for
some appropriate $n > 0$, $\Gamma_1,\ldots,\Gamma_n,\Gamma$ and
$\alpha$ so that $\Gamma = end(m \cap h)$. We will assume that in
fact $\alpha = \uparrow$, the reasoning for the case $\alpha =
\downarrow$ is similar. It follows then, by Lemma
\ref{intersections}.2:

\begin{align*}
\bigcap_{h' \in H_m}Act(m,h') = \{ z \} \cup \{ y_{(j, C)} \mid
Prove(j, C) \wedge \neg\Box &Prove(j, C)
    \in \Gamma_1 \} \cup\\
    &\cup \{ u_A \mid \Box Prove(j, C) \in \Gamma_1
    \}.
\end{align*}

We know that $B \notin \mathcal{E}(m,z)$ since $Proven(B) \notin
\Gamma$. If, for some $j \in Ag$, we would have $Prove(j, B) \in
\Gamma_1$, it would follow that $Proven(B) \in \Gamma$, since
$(\Gamma_1,\ldots,\Gamma_n,\Gamma)\alpha$ is an element.
Therefore, if $u_C, y_{(j,C)} \in \bigcap_{h' \in H_m}Act(m,h')$
for any $j \in Ag$, then $C \neq B$ and therefore $B \notin
\mathcal{E}(m, u_C) = \mathcal{E}(m, y_{(j,C)}) = \{ C \}$. It
follows then that for no proof which is presented under all
histories through $m$, this proof will be acceptable for $B$,
hence we get $\mathcal{M},m,h \not\models Proven(B)$.

\emph{Case 6}. $A = Prove(j,B)$ for some $j \in Ag$. Assume that
$Prove(j,B) \in end(m \cap h)$. Then we know that the length of
$m$ must be $1$. Indeed, if length of $m$ were greater than $1$,
then we would have $K\Box\neg Prove(j,B) \in end(m \cap h)$,
whence, by S4 reasoning for $K$, S5 reasoning for $\Box$, and
maxiconsistency of $end(m \cap h)$ in $X$ we would have $\neg
Prove(j,B) \in end(m \cap h)$, so that $Prove(j,B) \in end(m \cap
h)$ would be impossible.

So, for some appropriate $\Gamma$ and $\alpha$ we will have both
$m = [(\Gamma)\alpha]_\equiv$ and $m \cap h = (\Gamma)\alpha$. We
need to consider two subcases:

\emph{Case 6.1}. $\Box Prove(j,B) \in \Gamma$. Then, for all $h'
\in Choice^m_j(h)$, we will have, of course, $m \cap h' \equiv m
\cap h$ which means, by maxiconsistency and S5 reasoning for
$\Box$, that we will also have $\Box Prove(j, B) \in end(m \cap
h')$. This will mean that for all $h' \in Choice^m_j(h)$ we will
have either $u_B$ or $w_B$ in $Act(m,h')$ and we will have, of
course $B \in \mathcal{E}(m, u_B) = \mathcal{E}(m, w_B)$. Further,
by \eqref{A9} and maxiconsistency in $X$ of every $end(m \cap h')$
with $h' \in Choice^m_j(h)$ we know that also $KB \in m \cap h'$
for every such $h'$. Therefore, we know by Case 3 above that
either $\mathcal{M},m,h' \models u_B\co B$, or
$\mathcal{M},m,h'\models w_B\co B$ for every $h' \in
Choice^m_j(h)$.

Assume, further, that for some $s \in Pol$ we have
$\mathcal{M},m,h \models s\co B$. Then, in particular, we must
have $B \in \mathcal{E}(m,s)$. By definition of $\mathcal{E}$, $s$
cannot then be a proof variable of the form $u_C$, $w_C$, or
$y_{(j,C)}$ for any $j \in Ag$ and any formula $C$ different from
$B$. Moreover, $s$ cannot be $z$, since we have $Prove(j,B) \in
end(m \cap h)$ whence by maxiconsistency of  $end(m \cap h)$ in
$X$ and \eqref{A9}, $\neg Proven(B) \in end(m \cap h)$, so that,
again by maxiconsistency, $Proven(B) \notin end(m \cap h) \in m$,
which means, by the above definition of $\mathcal{E}$, that $B
\notin \mathcal{E}(m,z)$. Therefore, assuming that
$\mathcal{M},m,h \models s\co B$, $s$ can be either in $Pol_X$ or
in $Pol \setminus (Pol_X \cup Y \cup W \cup U \cup \{ z \})$, or
else in $\{ w_B, u_B, y_{(j,B)} \mid j \in Ag \}$. Well, if $s$ is
either in $Pol_X$ or in $Pol \setminus (Pol_X \cup Y \cup W \cup U
\cup \{ z \})$, then it is immediate from the definition of $Act$
that $s \notin Act(m,h)$. On the other hand, if $s \in  \{
y_{(j,B)} \mid j \in Ag \}$, then note that by maxiconsistency of
$end(m \cap h)$ in $X$ we must have $Prove(j,B) \wedge \neg\Box
Prove(j, B) \notin \Gamma$ whence it immediately follows that,
again $s \notin Act(m,h)$. Finally, consider two elements
$(\Gamma)\uparrow$ and $(\Gamma)\downarrow$. One of these elements
is actually
 $(\Gamma)\alpha$, both elements are in $m$, and, by Lemma \ref{hist}.2, we can choose $h', h'' \in H_m$ in such a way that we have both
 $ (\Gamma)\uparrow = m \cap h'$ and $ (\Gamma)\downarrow = m \cap h''$ . It clearly follows then from the definition of $Act$ that
 $w_B \notin Act(m, h')$, whereas $u_B \notin Act(m, h'')$.

 \emph{Case 6.2}. $\Box Prove(j,B) \notin \Gamma$. Then, by maxiconsistency of $\Gamma = end(m \cap h)$ in $X$,
 we must have $Prove(j,B) \wedge \neg\Box Prove(j,B) \in \Gamma$ as well as (again, by maxiconsistency of $\Gamma$  in $X$ and Lemma \ref{theorems}.4)  $[j](Prove(j,B) \wedge \neg\Box Prove(j,B)) \in \Gamma$.
 Therefore, for every $h' \in Choice^m_j(h)$  we will have $Prove(j,B) \wedge \neg\Box Prove(j,B) \in end(m \cap h')$ simply by definition of  $Choice$.
 This further means that for every such $h'$, the proof variable $y_{(j,B)}$ will be in $Act(m,h')$. Besides, it is immediate
 from the definition of $\mathcal{E}$ that $B \in \mathcal{E}(m, y_{(j,B)})$. Finally, note that by \eqref{A9} and maxiconsistency of the
 respective $end(m \cap h')$  in $X$, we will have $KB \in end(m \cap h')$ for every $h' \in Choice^m_j(h)$. Therefore, by Case 3 above, we will have
 $\mathcal{M},m,h'\models y_{(j,B)}\co B$ for every $h' \in Choice^m_j(h)$.

 Assume, further, that for some $s \in Pol$ we have $\mathcal{M},m,h \models s\co B$. Just as
 in Case 6.1, we can show that $s$ cannot be of the form $z$, $u_C$, $w_C$, or $y_{(j,C)}$ for any $j \in Ag$ and any formula $C$ different from $B$.
 Then, again borrowing our reasoning from the Case 6.1 above, we can show that
 if $s \in Pol_X$ or $s \in Pol \setminus (Pol_X \cup Y \cup W \cup U \cup \{ z \})$, then we must have
 $s \notin Act(m,h)$. If $s$ is $u_B$ or $w_B$ then we must have $s \notin Act(m,h)$ since $\Box Prove(j,B) \notin \Gamma = end(m \cap h)$,
 and therefore, by maxiconsistency of $\Gamma$ in $X$ and \eqref{A10} we must have $\Box Prove(i,B) \notin \Gamma = end(m \cap h)$ for all $i \in Ag$.
 Assume then that $s$ is $y_{(i,B)}$ for some $i \in Ag$. If $y_{(i,B)} \notin Act(m,h)$, then we are done. If, on the other hand,
 $y_{(i,B)} \in Act(m,h)$, then, by definition of $Act$, we must have $Prove(i,B) \wedge \neg\Box Prove(i,B) \in \Gamma = end(m \cap h)$, hence, by
 Lemma \ref{elementaryconsistency}.5, $\neg\Box Prove(i,B) \in end(m \cap h)$. Then the set
$$
\Xi = \{ \Box C \mid \Box C \in \Gamma \} \cup \{ \neg Prove(i,B) \}
$$
must be consistent, since otherwise we would have
$$
\vdash (\Box C_1\wedge\ldots\wedge\Box C_n) \to Prove(i,B)
$$
for some $\Box C_1,\ldots,\Box C_n \in \Gamma$, whence, since
$\Box$ is an S5-modality, we would get
$$
\vdash (\Box C_1\wedge\ldots\wedge\Box C_n) \to \Box Prove(i,B),
$$
which would mean that $\Box Prove(i,B) \in \Gamma$, contrary to
our assumption. Therefore, $\Xi$ is consistent and we can extend
$\Xi$ to a set $\Delta \subseteq Form_X$ which is maxiconsistent
in $X$. Of course, in this case $Prove(i,B) \notin \Delta$. We
will have then that $(\Delta)\alpha$ is an element, and, by
definition of $\equiv$, that $(\Gamma)\alpha \equiv
(\Delta)\alpha$. By Lemma \ref{hist}.2, for some $h' \in H_m$ we
will have $(\Delta)\alpha = m \cap h'$ and, therefore, $\Delta =
end(m \cap h')$. Since $Prove(i,B) \notin \Delta$, it follows that
$y_{(i,B)} \notin Act(m,h')$.

Thus we have shown that if $Prove(j,B) \in end(m \cap h)$, then
$\mathcal{M},m,h\models Prove(j,B)$. For the inverse direction,
assume that $Prove(j,B) \notin end(m \cap h)$. Again, we have to
consider two further subcases:

\emph{Case 6.3}. The length of $m$ equals $1$ so that, for some
appropriate $\Gamma$ and $\alpha$ we have both $m =
[(\Gamma)\alpha]_\equiv$ and $m \cap h = (\Gamma)\alpha$. If
$\mathcal{M},m,h \models Proven(B)$, then by \eqref{A9} we will
have $\mathcal{M},m,h \not\models Prove(j,B)$, and thus we will be
done. Therefore, assume that $\mathcal{M},m,h \not\models
Proven(B)$. Moreover, if $\mathcal{M},m,h \not\models KB$ then we
will again have, by \eqref{A9}, that $\mathcal{M},m,h \not\models
Prove(j,B)$, so that we may also safely assume that
$\mathcal{M},m,h \models KB$. Under these assumptions, in order to
show that $\mathcal{M},m,h \not\models Prove(j,B)$ we have to show
that the positive condition fails in that there is an $h' \in
Choice^m_j(h)$  such that no acceptable proof of $B$ is present in
$Act(m,h')$. To this end, we consider the set
$$
    \Xi = \{ [j]C \mid [j]C \in \Gamma \} \cup \{ \bigwedge_{i \in Ag}\neg Prove(i,B) \}.
$$
This set must be consistent, since otherwise we would have
$$
\vdash ([j]C_1\wedge\ldots\wedge[j] C_n) \to \bigvee_{i \in Ag}Prove(i,B)
$$
for some $[j] C_1,\ldots,[j]C_n \in \Gamma$, whence, since $[j]$
is an S5 modality, we would get
$$
\vdash ([j]C_1\wedge\ldots\wedge[j]C_n) \to [j](\bigvee_{i \in Ag}Prove(i,B)),
$$
which would mean that $[j](\bigvee_{i \in Ag}Prove(i,B)) \in \Gamma$. On the other hand,
since

\noindent$Prove(j,B) \notin \Gamma$, this means, by
maxiconsistency of $\Gamma$ in $X$, that $\neg Prove(j,B) \in
\Gamma$, whence, again by maxiconsistency and \eqref{A13}, we
obtain that $\langle j\rangle(\bigwedge_{i \in Ag}\neg Prove(i,B))
\in \Gamma$. Therefore, by maxiconsistency of $\Gamma$ in $X$, we
must have $\neg [j](\bigvee_{i \in Ag}Prove(i,B)) \in \Gamma$, a
contradiction.

Therefore, $\Xi$ is consistent and we can extend $\Xi$ to a set
$\Delta \subseteq Form_X$ which is maxiconsistent in $X$. Of
course, in this case we will have $Prove(i,B) \notin \Delta$ for
all $i \in Ag$. We will have then that $(\Delta)\alpha$ is an
element, and, arguing as in Case 2.1 we can show \eqref{E:t1} so
that $\Delta$ contains all boxed formulas from $\Gamma$.
Therefore, by definition of $\equiv$, we know that $(\Gamma)\alpha
\equiv (\Delta)\alpha$. By Lemma \ref{hist}.2, we know that for
some $h' \in H_m$ we will have $(\Delta)\alpha = m \cap h'$ and,
therefore, $\Delta = end(m \cap h')$. Also, since $\Delta$
contains all the $[j]$-modalized formulas from $\Gamma$, we know
that for any such $h'$ we will have $h' \in Choice^m_j(h)$. We
also know that $Proven(B) \notin \Delta$, for otherwise we would
have, by maxiconsistency of $\Delta$ and Lemma \ref{theorems}.2,
that $\Box Proven(B) \in \Delta$, whence, by the fact that
$(\Gamma)\alpha \equiv (\Delta)\alpha$ we would have that
$Proven(B) \in \Gamma$, contradicting our assumptions.

Consider then $Act(m,h')$. We may assume that
$\alpha = \uparrow$, the reasoning for the case when $\alpha = \downarrow$
is similar. We have, by definition of $Act$ that:
$$
Act(m,h') = \{ z \} \cup \{ y_{(i, C)} \mid Prove(i, C) \wedge \neg\Box Prove(i, C)
    \in \Delta \} \cup \{ u_C \mid \Box Prove(i, C) \in \Delta
    \}.
$$
We know that $B \notin \mathcal{E}(m, z)$, since we have
established that $Proven(B) \notin \Delta$; we also know that if
$u_C, y_{(i, C)} \in Act(m,h')$ for any $i \in Ag$, then $C \neq
B$ since for all $i \in Ag$ we have $Prove(i,B) \notin \Delta$,
and this means that if $u_C, y_{(i, C)} \in Act(m,h')$ for any $i
\in Ag$, then both $B \notin \mathcal{E}(m, u_C)$ and $B \notin
\mathcal{E}(m, y_{(i, C)})$. Therefore, at $(m,h')$ there exists
no presented proof which would be acceptable for $B$, and since
$h' \in Choice^m_j(h)$, this means that the positive condition for
$Prove(j,B)$ at $(m,h)$ is violated, so that we get
$\mathcal{M},m,h \not\models Prove(j,B)$ as desired.

\emph{Case 6.4}. The length of $m$ is greater than $1$. Then, by
Lemma \ref{intersections}.2, for all $h' \in H_m$ we have that
$$
Act(m,h') = \bigcap_{h'' \in H_m}Act(m,h'').
$$
Assume then, that we have both $s \in Act(m,h)$ and $\mathcal{M},m,h \models s\co B$ for some $s \in Pol$. Then $s \in \bigcap_{h'' \in H_m}Act(m,h'')$,
which means that the negative condition for $Prove(j,B)$ at $(m,h)$ is violated and we must have
$\mathcal{M},m,h \not\models Prove(j,B)$. Assume, on the contrary, that there is no
$s \in Pol$ for which both $s \in Act(m,h)$ and $\mathcal{M},m,h \models s\co B$.
 Then, since $h$ is of course in $Choice^m_j(h)$, it turns out that the positive condition
 for $Prove(j,B)$ at $(m,h)$ is violated and again have
$\mathcal{M},m,h \not\models Prove(j,B)$. So, in any case $\mathcal{M},m,h \not\models Prove(j,B)$,
as desired.

This finishes the list of the modal induction cases at hand, and
thus the proof of our truth lemma is complete.
\end{proof}

\section{The main result}\label{main}

We are now in a position to prove Theorem \ref{completeness}. The
proof proceeds as follows. One direction of the theorem was proved
as Corollary \ref{c-soundness}. In the other direction, assume
that $\Gamma \subseteq Form_X$ is consistent. Then, by Lemma
\ref{elementaryconsistency}.1, $\Gamma$ can be extended to a
$\Delta$ which is maxiconsistent in $X$. But then choose an
arbitrary $\alpha \in \{ \uparrow, \downarrow \}$ and consider
$\mathcal{M} = \langle Tree, \leq, Choice, Act, R, \mathcal{E},
V\rangle$, the canonical model defined in Section
\ref{canonicalmodel}. The structure $(\Delta)\alpha$ is an
element, therefore $[(\Delta)\alpha]_\equiv \in Tree$. By Lemma
\ref{hist}.2, there is a history $h \in H_m$ such that
$(\Delta)\alpha = [(\Delta)\alpha]_\equiv \cap h$. For this $h$,
we will also have $\Delta = end([(\Delta)\alpha]_\equiv \cap h)$.
By Lemma \ref{truth}, we therefore get that:
$$
\mathcal{M}, [(\Delta)\alpha]_\equiv, h \models \Delta \supseteq
\Gamma,
$$
and thus $\Gamma$ is shown to be satisfiable in a normal jstit
model.

\textbf{Remark}. Note that the canonical model used in this proof
is $X$-universal in the sense that it satisfies every consistent
subset of $Form_X$.

As an obvious corollary of Theorem \ref{completeness} we get the
following weak completeness result:

\begin{corollary}\label{weakcompleteness}
For every $A \in Form$, $\vdash A$ iff $A$ is valid over normal
jstit models.
\end{corollary}
\begin{proof}
One direction follows from Theorem \ref{soundness}. In the other
direction, if $\not\vdash A$, then $\{ \neg A \}$ is consistent.
Setting $X$ to be the set of proof variables occurring in $A$, we
see that $PVar \setminus X$ must be countably infinite. Therefore,
Theorem \ref{completeness} applies, $\{ \neg A \}$ must be
satisfied in some normal jstit model, and $A$ cannot be valid.
\end{proof}
As a further corollary, we deduce a restricted form of compactness
property:
\begin{corollary}\label{compactness}
Let $X \subseteq PVar$ be such that $PVar \setminus X$ is
countably infinite. Then an arbitrary $\Gamma \subseteq Form_X$ is
satisfiable iff every finite $\Gamma_0 \subseteq \Gamma$ is
satisfiable.
\end{corollary}
\begin{proof}
If $\Gamma$ is satisfiable, then clearly every finite $\Gamma_0
\subseteq \Gamma$ is satisfiable. On the other hand, if every
finite $\Gamma_0 \subseteq \Gamma$ is satisfiable, then for no
$A_1,\ldots, A_n \in \Gamma$ can we have that $\vdash (A_1
\wedge\ldots \wedge A_n) \to \bot$, for otherwise, by Theorem
\ref{soundness}, the finite set $\{ A_1,\ldots, A_n \}$ would be
unsatisfiable. Therefore, $\Gamma$ must be consistent, and, by
Theorem \ref{completeness}, also satisfiable.
\end{proof}

\section{Conclusions and future research}\label{conclusion}

Theorem \ref{completeness}, the main result of this paper, proves
what might be called a restricted strong completeness theorem for
the implicit jstit logic. As we have shown in Section \ref{main},
this means, among other things, that this logic allows for a
finitary proof system and enjoys a restricted form of compactness
property. Taken together, these results show that, given the rich
variety of expressive means present in the implicit jstit logic
and non-trivial semantic constraints imposed on its models, this
logic displays a surprising degree of regularity.

Of course, the results of the present paper give room to some
generalization. One obvious observation would be that the rule
\eqref{R3} gives but one variant out of the infinite family of the
so-called \emph{constant specifications} allowed for in
justification logic; and it is straightforward to see that the
above completeness proof can be easily adapted for the systems
with other versions of constant specification. The other obvious
direction of generalizing the results above would be to relieve
the restriction that $R = R_e$ and consider the semantics of
\cite{OLWA} in its full generality, although, as we have already
mentioned, it is not so clear whether this generalization will
affect the set of validities.

In the broader perspective, Theorem  \ref{completeness} is a step
towards axiomatization of the full basic justification stit logic
in case such an axiomatization is possible. Viewing Theorem
\ref{completeness} as a partial success in axiomatizing the full
basic jstit logic, it is easy to see which steps shall come next.
First, one needs to understand the mechanics behind the proving
modalities omitted from the implicit jstit logic and axiomatize
the logic of $Prove(j,t,A)$ and $Proven(t,A)$ placed on top of
stit and justification modalities; then an axiomatization of a
system combining both explicit and implicit proving modalities and
their interplay may turn out to be possible. As a promising
further step in this direction, one can consider, for example, the
logic of the so-called $E$-notions, introduced in \cite{OLWA2}. It
allows one to define a combination of implicit and explicit
proving modalities, even though this combination is but a subset
of the variety of proving modalities definable within the full
basic jstit logic, and can, therefore, provide a demo version of
the problems to be encountered in an attempt to explore the
properties of the full system.

\section{Acknowledgements}
To be inserted.

}

\end{document}